\documentclass[11pt,oneside]{article}
\usepackage[utf8]{inputenc}
\usepackage[left=2cm,right=2cm,top=3.2cm,bottom=3.2cm]{geometry}
\usepackage{graphicx,caption}
\usepackage{latexsym,booktabs}
\usepackage{blindtext,textcomp}
\usepackage{amsthm, amsmath, amssymb, amsbsy,mathtools}
\usepackage{hyperref}
\usepackage{tikz}
\usepackage{authblk}
\usepackage{enumerate}
\usepackage{amsopn}
\usepackage{authblk}
\usepackage[
    left = `,%
    right = ',%
    leftsub = ``,%
    rightsub = ''%
]{dirtytalk}
\usepackage{sectsty}
\usepackage{mathtools}
\usepackage[autostyle]{csquotes}
\usepackage{comment}
\allowdisplaybreaks
\numberwithin{equation}{section}

\newtheorem{theorem}{Theorem}[section]
\newtheorem{remark}[theorem]{Remark}
\newtheorem{cor}[theorem]{Corollary}
\newtheorem{pro}[theorem]{Proposition}
\newtheorem{lemma}[theorem]{Lemma}
\newtheorem{defn}[theorem]{Definition}
\newtheorem{ex}[theorem]{Example}

\def \H {{\mathcal H}}
\def \M {{\mathcal M}}
\def \W {{\mathcal W}}

\title{Invariant subspaces of powers of some unicellular operators}
\author[1]{Sneh Lata\thanks{Corresponding Author: sneh.lata@snu.edu.in}}
\author[1]{Sushant Pokhriyal\thanks{sp259@snu.edu.in}}
\author[2]{Dinesh Singh\thanks{dineshsingh1@gmail.com}}
\affil[1]{Department Of Mathematics\\
         Shiv Nadar University\\
         School of Natural Sciences\\
         Gautam Budh Nagar - 203207\\
         Uttar Pradesh, India}
\affil[2]{Centre For Lateral Innovation, Creativity and Knowledge\\
        SGT University\\
        Gurugram 122505\\
        Haryana, India}
\date{}

\begin{document}
\maketitle

\begin{abstract}
In this paper we study subspaces which are invariant under squares and cubes (separately as well as jointly) of unicellular backward weighted shift operators on a separable Hilbert space. The finite-dimensional subspaces are characterized for all weights and the infinite-dimensional subspaces are characterized for two classes of weights. 

\end{abstract}



\section{Introduction}
The theory of operators on a Hilbert space has to a significant extent revolved around the problem of characterizing the invariant subspaces of specific examples of operators. The primary motives are two-fold for such an approach. In the first 
instance several interesting mathematical results follow from the study of the structure of invariant subspaces of specific 
operators. A classical illustration of this assertion is the fundamental work of Beurling \cite{beurling1949two} in describing the invariant 
subspaces of the unilateral shift on the Hardy space $H^2.$ Another instance of importance is the description of the invariant subspaces of a weighted shift operator on the Hilbert space of square summable sequences by Donoghue \cite{donoghue1957lattice} where the weight sequence was $\{\frac{1}{2^n}\}_{n=0}^\infty$. Donoghue's work gave rise to a series of deep and interesting generalizations by Nikolskii and others about which we shall elaborate a little in the coming paragraphs. 
The second motivation for the study of invariant subspaces is the fact that insights tend to emerge about an operator when it is restricted to an invariant subspace and the operator then assumes a simple form as can be seen in the case of the Wold decomposition of an isometry \cite[page 109]{hoffman2014banach} and in recent articles \cite{LATA2022126184, article}  by the authors of this paper.  

Throughout this article, $\mathbb{N}$ denotes the set of non-negative integers and $\H$ denotes a separable Hilbert space with orthonormal basis $\{e_{n}\}_{n\in \mathbb{N}}$. For a bounded sequence $\{w_{n}\}_{n\in \mathbb{N}}$ of positive real numbers, the \textit{forward weighted shift} with the weight sequence  
$\{w_n\}_{n\in \mathbb N}$ is the operator $T\in B(\H)$ defined by 
$T(e_{n}) = w_{n}e_{n+1}$ for all $n\in \mathbb N$ and its adjont, known as the \textit{backward weighted shift}, is given by  $T^*(e_0)=0$ and $T^*(e_n) = w_{n-1}e_{n-1}$ for all $n\ge 1.$  

In this article, we are interested in unicellular weighted shift operators. These are weighted shifts for which the lattice of invariant subspaces is a totally ordered set with respect to the set inclusion. In \cite{donoghue1957lattice}, Donoghue gave the first example of a unicellular weighted shift with $\H = \ell{^2}$ and $w_{n}=\frac{1}{2^n}$ where he showed that the only invariant subspaces for this weighted shift are of the form $\bigvee_{n\geq k}\{e_n\}$ for $k\in \mathbb{N} $. Identifying the key properties of these weights, Nikolskii extended this result to the entire class of square summable monotonically decreasing sequence of positive real numbers \cite{nikolskii1965invariant}. In fact, he proved much more, he showed that the lattice of invariant subspaces of $T$ stays the same even when you consider it as an operator on $\ell{^p}$ for $1\leq p\leq \infty$ and $\{w_n\}_{n\in \mathbb{N}}$ to be a monotonically decreasing sequence in $\ell{^q}$ consisting of positive numbers where $1/p+1/q=1.$ He continued this line of research in \cite{nikol1967invariant, nikol1968basicity} and gave few more necessary and sufficient conditions on weights for $T$ to be unicellular. In \cite{yakubovich1985invariant}, Yakubovich proved unicellularity of $T$ by obtaining the same lattice of invariant subspaces under much weaker conditions on the weights. We refer the interested readers to \cite{harrison1971unicellularity,herrero1990unicellular,kang1992unicellularity, kang1994study, joo1995unicellularity, radjavi2003invariant, shields1974weighted, yadav1982invariant,  yadav1982characterization, yousefi2001unicellularity} and the references therein.  

In this paper, we mainly focus on three problems: $(i)$ characterization of invariant subspaces of $T^{*2},$  $(ii)$ characterization of invariant subspaces of $T^{*3},$ and $(iii)$ characterization of joint-invariant subspaces of $T^{*2}$ and $T^{*3}.$ Characterizing invariant subspaces for an operator is equivalent to characterizing invariant subspaces for its adjoint; so our work is also a step in the above-mentioned direction of research.  

Our investigation deals with finite-dimensional and infinite-dimensional subspaces separately. For the finite-dimensional case, we work with the assumption that $T$ is a unicellular operator. In this case, we show that, with simple manipulations, all weights can assumed to be 1.  This allows us to obtain characterizations of finite-dimensional subspaces in above-mentioned all three problems under the assumption that $T$ is unicellular without imposing any extra conditions on the weights. On the contrary, the nature of weights play a crucial role in study of infinite-dimensional subspaces. We give characterizations of infinite-dimensional subspaces in all above-mentioned problems for two classes of weights. Here we do
not assume T to be unicellular; however, for both these classes
T is already known to be unicellular. The invariant subspaces of T for these
classes of weights have been characterized by Nikolskii in  \cite{nikolskii1965invariant} and Yadav \& Chaterjee in  \cite{yadav1982characterization} from where the unicellularity of $T$ for these classes of weights follows as a by-product. But we want to note that the case of finite-dimensional subspaces is far more daunting as compared to the infinite-dimensional subspaces case. 

The organization of the paper is as follows. In Section \ref{sec 2}, we give characterizations of finite-dimensional subspaces that are invariant under $T^{*2}$ (Theorem \ref{thm 2.5}),  $T^{*3}$ (Thereom \ref{ thm 2.8}), and joint-invarinat under $T^{*2}$ and $T^{*3}$ (Theorem \ref{thm 2.10}). In Section \ref{sec 3}, we consider T for two classes of weights. The invariant
subspaces for these classes of weights have been characterized by Nikolskii in  \cite{nikolskii1965invariant} and Yadav \& Chaterjee in  \cite{yadav1982characterization}. In Theorem \ref{thm 3.6}, we show that a condition on weights considered in (\cite{yadav1982characterization}, Theorem 1)  is redundant for characterization of invariant subspaces of $T$. In this same paper, authors claim that their result generalizes the main theorem of \cite{nikolskii1965invariant}, but our Examples \ref{ex 3.7} \& \ref{ex 3.8} show that the set of conditions on weights in \cite{nikolskii1965invariant} and \cite{yadav1982characterization} are independent of each other. We then give characterizations of infinite-dimensional subspaces that are invariant under $T^{*2}$ (Theorem \ref{thm 3.9}), invariant under $T^{*3}$ (Theorem \ref{thm 3.11}), and joint-invariant under $T^{*2}$ and $T^{*3}$ (Theorem \ref{thm 3.12}) corresponding to these two classes of weights. In the last section, we give some remarks and open problems related to $T$. We end this section with an interesting result on quasinilpotent unicellular operators. We give a sufficient condition (Corollary \ref{cor 4.4}) on analytic functions
$f$ so that $f(A)$ stays unicellular for a quasinilpotent unicellular operator $A$. This result not only generalizes the results of \cite{kang1994study} and \cite{joo1995unicellularity} for a much bigger class of unicellular operators but also gives a simpler proof of their results.

\section{Finite-dimensional subspaces }\label{sec 2}
Through out this section we will assume $T$ to be a unicellular operator. In this section we will characterize finite-dimensional subspaces of $\H$ that are invariant under $T^{*l}$ for $2\le l\le 3,$ and we will also give a characterization of finite-dimensional subspaces that are jointly invariant under $T^{*2}$ and 
$T^{*3}$. We start by recording a very useful and crucial observation that if a finite-dimensional subspace is invariant under $T^{*l}$ for some $l,$ then it must be contained in $M_k=\vee_{i=0}^{k}e_{i}$ for some $k\in \mathbb{N}.$ To justify this, we first note that 
in \cite{fang1992rosenthal} the author proved that a unicellular forward weighted shift on a separable Hilbert space is always quasinilpotent, that is, its spectrum equals $\{0\}.$ In particular, $\sigma(T^*)=\{0\}$ which yields that $\sigma(T^{*l})=\{0\}$ for every $l.$ 
Thus the resolvent set $\rho({T^{*l}})=\mathbb{C}\setminus \{0\}$ does not has any bounded component;  
therefore, the full spectrum of 
$T^{*l}$ also equals $\{0\}.$ Recall that for any bounded linear operator $A$ on a Hilbert space, the full spectrum of 
$A,$ denoted by $\eta(\sigma(A))$, is defined as the union of $\sigma(A)$ and the bounded components of $\rho(A).$  Now let $\M$ be a closed subspace of $\H$ that is invariant under $T^{*l}$ for some $l.$ Then, using Theorem 0.8 from \cite{radjavi2003invariant}, $\sigma(T^{*l}|\M) \subseteq  \eta(\sigma(T^{*l})) = \{0\}$ which gives that $\sigma(T^{*l}|\M)=\{0\}.$ Now, if we further assume $\M$ to be finite-dimensional, then 
$\sigma(T^{*l}|\M)=\{0\}$ which makes $T^{*l}|\M$ a nilpotent operator. As a result, there exists some $n$ such that $T^{*ln}(x)= 0$ for every $x\in \M$ which establishes that $\M\subseteq M_{nl-1.}$ 

Next we show that in case of finite-dimensional subspaces it is enough to focus on weighted shifts where all weights equal to 1. For a fixed $k\geq 1,$ define
\begin{align*}
    X &:M_{k}\rightarrow{M_{k}}\quad \text{by}\quad Xe_{n}= \delta_{n}e_{n} \ \text{with}~ \delta_{0}=1, ~\delta_{n}=w_{0}w_{1}\cdots w_{n-1} ~\text{for} ~1\leq n\leq k;\\
    \text{and}\quad T_{1}^{*} &:M_{k}\rightarrow{M_{k}}\quad \text{by}\quad T_{1}^{*}e_{0}=0, \ T_{1}^{*}e_{n}=e_{n-1} ~ \text{for} ~ 1\leq n\leq k.
\end{align*}

Then $(X^{-1}T_{1}^{*}X)e_{n} = (X^{-1}T_{1}^{*})\delta_{n}e_{n}
    = \delta_{n}X^{-1}e_{n-1}
    =\frac{\delta_{n}}{\delta_{n-1}}e_{n-1}
    = w_{n-1}e_{n-1}
    = T^{*}e_{n}$. Therefore, $X^{-1}T_{1}^{*}X=T^{*}$ which implies that $X^{-1}T_{1}^{*r}X=T^{*r}$ for any $r\geq 1$. So if $\M\subseteq M_{k}$ is a closed subspace, then $T_{1}^{*r}(\M)\subseteq \M$ if and only if $T^{*r}(X^{-1}\M)\subseteq X^{-1}\M$. This precisely means that for characterization of invariant subspaces of $T^{*r}$ for any $r\geq 1$, it is enough to characterize invariant subspaces of $T_{1}^{*r}$.
Hence, throughout this section we will work with the backward shifts with all the weights equal to $1$ and for convenience of notation we will denote $T_{1}^{*}$ by $T^{*}$.
 
\begin{flushleft}
 The following are two well-established results pertaining to nilpotent operators.
\end{flushleft}
\begin{lemma}\label{lem 2.1}
Let $V$ be a vector space and $S:V\rightarrow{V}$ be a linear operator. If for some $x\in V$ and some positive integer $m$ 
$$
S^{m-1}x\neq 0~~\text{but}~~S^{m}x=0,
$$
then the set $\{x,Sx,\ldots,S^{m-1}x\}$ is linearly independent.
\end{lemma}
  
\begin{theorem}{\textbf{(Cyclic Nilpotent Theorem).}}\label{thm 2.2}
Every linear nilpotent operator on a finite-dimensional vector space splits into a direct sum of cyclic linear nilpotent operators.
\end{theorem}
The above theorem asserts that if $V$ is a finite-dimensional vector space and $S:V\rightarrow{V}$ is a linear nilpotent operator, then there exist $x_{1},\ldots,x_{k}\in V$ and $ m_{1},\ldots,m_{k}\in \mathbb{N}$ such that
$$
V=\bigoplus_{i=1}^{k} ~\langle x_{i}\rangle_{S},~\text{where}~\langle x_{i}\rangle_{S}=\text{span}~\{x_{i}, Sx_{i},\ldots, S^{m_{i}}x_{i}\}~\text{for some}~ m_{i}\quad \text{and} \quad\text{dim}~V=\sum_{i=1}^{k}(m_{i}+1).
$$

\begin{flushleft}
 Using Lemma \ref{lem 2.1} and Theorem \ref{thm 2.2}, we obtain the following result.
\end{flushleft}
\begin{cor}\label{cor 2.3}
Let $\M$ be a finite-dimensional subspace of $\H$ such that $T^{*n}\M\subseteq \M$ for some $n\in \mathbb{N}$. Then $\M$ splits into direct sum of at most $n$ number of $T^{*n}$- invariant cyclic subspaces, that is, there exist
$x_{1},\ldots,x_{k}\in \M, 1\leq k\leq n$ such that
  $$
\M=\bigoplus_{i=1}^{k} ~\langle x_{i}\rangle_{T^{*n}}.
  $$
\end{cor}
\begin{proof}
Since $T^{*n}$ is a nilpotent operator on $\M$, therefore $\M\subseteq M_{k}$ for some $k\in \mathbb{N}$. Then, according to Theorem \ref{thm 2.2}, there exist vectors $x_{1},\ldots,x_{k}\in \M$ such that
$$
\M=\bigoplus_{i=1}^{k} ~\langle x_{i}\rangle_{T^{*n}},~\text{where}~\langle x_{i}\rangle_{T^{*n}}=\text{span}~\{x_{i}, T^{*n}x_{i},\ldots, T^{*nm_{i}}x_{i}\}\quad\text{and}\quad\text{dim}~\M=\sum_{i=1}^{k}(m_{i}+1).
$$
Since $T^{*n(m_{i}+1)}x_{i}=0$ for all $i$, therefore $\{T^{*nm_{i}}x_{i}\}_{i}\subseteq M_{n-1}$. Also dim $M_{n-1}=n$ and the set $\{T^{*nm_{i}}x_{i}\}_{i}$ is linearly independent, this implies that $1\leq k\leq n$. This completes the proof.
\end{proof}

We first consider the case of finite-dimensional subspaces that are invariant under $T^{*2}.$ From the above discussion we know that if $\M$ is finite-dimensional subspace such that $T^{*2}\M\subseteq \M$, then $\M\subseteq M_{k}$ for some $k\in \mathbb{N}$. Suppose dim $\M=n,$ then we can easily verify that:

\begin{enumerate}[(i)]
\item $\M$ is non-cyclic if and only if $\M\subseteq M_{k}$ for some $n-1\leq k\leq 2n-3$.
    \item  $\M$ is cyclic if and only if $\M\subseteq M_{k}$ for some $ 2n-2\leq k\leq 2n-1$.
\end{enumerate}

\begin{lemma}\label{lem 2.4}
 Let $\M\subseteq \H$ be a finite-dimensional non-cyclic $T^{*2}$- invariant subspace. Then $e_{0}, e_{1}\in \M$.
\end{lemma}
\begin{proof}
Since $\M$ is a finite-dimensional non-cyclic subspace of $\H$ invariant under $T^{*2}$, then using Corollary \ref{cor 2.3}, there exists $x_{1}, x_{2}\in \M$ such that
$$
 \M= \langle x_{1}\rangle_{T^{*2}} \oplus \langle x_{2}\rangle_{T^{*2}},
$$
where $\langle x_{i}\rangle_{T^{*2}}=\text{span}~\{x_{i}, T^{*2}x_{i},\ldots, T^{*2m_{i}}x_{i}\}$ for $i=1, 2$, and $\sum_{i=1}^{2}(m_{i}+1)=\text{dim}~\M$.

Since $T^{*2(m_{1}+1)}x_{1}=0$ and $T^{*2(m_{2}+1)}x_{2}=0$, therefore $\{T^{*2m_{1}}x_{1}, T^{*2m_{2}}x_{2}\} \subseteq M_{1}$. Now, let us consider the following
$$
T^{*2m_{1}}x_{1} = ae_{0}+be_{1}\quad\text{and}\quad
    T^{*2m_{2}}x_{2} = ce_{0}+de_{1},\quad\text{where}~a,b,c,d \in \mathbb{C}.
$$

Using the above equations we can deduce that $dT^{*2m_{1}}x_{1}-bT^{*2m_{2}}x_{2}= (ad-bc)e_{0}$. We know the set $\{T^{*2m_{1}}x_{1}, T^{*2m_{2}}x_{2}\}\subseteq \M$, therefore $e_{0}\in \M$. Similarly we can prove that $ e_{1}\in \M$, and we are done.
\end{proof}
\begin{theorem}\label{thm 2.5}
Let $\M$ be a non-cyclic $n$-dimensional subspace of $\H$ invariant under $T^{*2}$, then
$$
\boxed{ \M= \text{span}\{e_{0},\ldots, e_{n-p-2},T^{*2p}x,\ldots, x \}}~,
$$
where $x\in \M$ such that $\langle x,e_{n+p}\rangle \neq 0$ for some $-1\leq p\leq n-3$.
\end{theorem}
\begin{proof}
Since $\M$ be a non-cyclic $n$-dimensional subspace of $\H$ invariant under $T^{*2}$, then $\M\subseteq M_{k}~(n-1\leq k\leq 2n-3)$. 

Let $k=2j+t$ for some $j\in \mathbb{N}~\text{and}~ t=0,1$. Therefore $\M\subseteq M_{2j+t}$ and there exist a $x\in \M$ such that $\langle x,e_{2j+t}\rangle \neq 0$. Consider
   \begin{equation}\label{eq 2.1}
       x= \sum_{i=0}^{2j+t}\alpha_{i}e_{i},\quad\alpha_{2j+t}\neq 0.
   \end{equation}
According to Lemma \ref{lem 2.1} the set $ \langle x\rangle_{T^{*2}}= \text{span}~\{x, T^{*2}x ,\ldots, T^{*2j}x\}$ is linearly independent and hence, dim $\langle x\rangle_{T^{*2}}=j+1$.

Since $\M$ is non-cyclic, then according to Corollary \ref{cor 2.3} there exist a subspace $ \W\subseteq \M$ invariant under $T^{*2}$ such that $\M=\langle x\rangle_{T^{*2}} \oplus \W$. 
Now, $\text{dim}~ \W= \text{dim}~ \M- \text{dim}~\langle x\rangle_{T^{*2}}=n-j-1$. This yields that $ \W\subseteq M_{2n-2j-3} $, otherwise if there exist an $v\in \W$ such that $\langle v,e_{i}\rangle \neq 0$ for some $i\geq 2n-2j-2$, then the set $\{y, T^{*2}y ,\ldots, T^{*2(n-j-1)}y\} \subseteq \W$ implying that dim $\W\geq n-j$. Suppose $\W=\langle y\rangle_{T^{*2}}$ for some $y\in M_{2n-2j-3}$, where
\begin{equation}\label{eq 2.2}
       y= \sum_{i=0}^{2n-2j-3}\beta_{i}e_{i},\quad\text{at least one the coefficients }\beta_{2n-2j-4},\beta_{2n-2j-3}\neq 0.
   \end{equation} 
   
 For $0\leq m\leq n-j-2$, the set $\{T^{*2(j-m)}x, T^{*2(n-j-2-m)}y\}\subseteq M_{2m+1}$. Now, we are going to apply induction on the set $\{ T^{*2(j-m)}x, T^{*2(n-j-2-m)}y\}$ for every $m$, to show that $\{e_{0}, e_{1},\ldots, e_{2n-2j-3}\}\subseteq \M$.
 
\vspace{.2cm}
 For $m=0$. The vectors $\{ T^{*2j}x, T^{*2(n-j-2)}y\}\subseteq M_{1}$ are linearly independent. Then according to Lemma \ref{lem 2.4} we have $e_{0}, e_{1}\in \M$.

\vspace{.2cm}
  For $m=l$. Let us assume that $\{e_{0}, e_{1},\ldots, e_{2l+1}\}\subseteq \M$.

\vspace{.2cm}
For $m=l+1$. We have $\{ T^{*2(j-l-1)}x, T^{*2(n-j-l-3)}y\}\subseteq M_{2l+3}$. Using Equations (\ref{eq 2.1}) and (\ref{eq 2.2}) we get
$$
T^{*2(j-l-1)}x = \sum_{i=0}^{2l+3}\alpha_{i}e_{i}\quad\text{and}\quad
            T^{*2(n-j-l-3)}y = \sum_{i=0}^{2l+3}\beta_{i}e_{i}.
$$

Let us consider the following
\begin{align*}
p_{1}&=T^{*2(j-l-1)}x-\sum_{i=0}^{2l+1}\alpha_{i}e_{i}=\alpha_{2l+2}e_{2l+2}+\alpha_{2l+3}e_{2l+3}\\
\text{and}\quad q_{1}&=T^{*2(n-j-l-3)}y - \sum_{i=0}^{2l+1}\beta_{i}e_{i}= \beta_{2l+2}e_{2l+2}+\beta_{2l+3}e_{2l+3}.
\end{align*}

For some $\alpha,\beta\in \mathbb{C}$, let us assume that $\alpha p_{1} +\beta q_{1}=0$ . Then applying $ T^{*2(l+1)}$ on both sides of the equation we get
     $$
     \alpha T^{*2j}x+\beta T^{*2(n-j-2)}y=0.
     $$
     
Since the set $\{T^{*2j}x,T^{*2(n-j-2)}y\}$ is linearly independent, hence $\alpha, \beta=0$. Thus $\{p_{1}, q_{1}\}\subseteq \M$ is linearly independent, therefore solving the above equations gives us that $e_{2l+2}, e_{2l+3}\in \M$. So by induction we have $\{e_{0}, e_{1},\ldots, e_{2n-2j-3}\}\subseteq \M$.\\

Now we will divide the rest of the proof in two cases as follows:\\

\noindent\underline{\textbf{Case $1$ ($t=0$).}} In this case $k=2j$. Using Equation (\ref{eq 2.1}) we get
$$
        T^{*2(2j-n+1)}x = \sum_{i=0}^{2n-2j-2}\alpha_{i}e_{i}\quad\text{for some}~\alpha_{i}\in \mathbb{C}.
$$

Then
$$
 T^{*2(2j-n+1)}x - \sum_{i=0}^{2n-2j-3}\alpha_{i}e_{i} = \alpha_{2n-2j-2}e_{2n-2j-2}\quad\text{which implies that}~  e_{2n-2j-2} \in \M.
$$

Since $\{e_{0}, e_{1},\ldots, e_{2n-2j-2}\}\subseteq \M$ and the set $\{x, T^{*2}x ,\ldots, T^{*2(2j-n)}x\}\subseteq \M$ is linearly independent, therefore $\M= \text{span}\{e_{0},\ldots, e_{2n-2j-2},T^{*2(2j-n)}x,\ldots, x \}$. Finally substituting the value of $n+p=2j$, we get
$$
 \boxed{\M= \text{span}\{e_{0},\ldots, e_{n-p-2},T^{*2p}x,\ldots, x \}} ~,
$$
where $\langle x,e_{n+p}\rangle \neq 0$ for some $-1\leq p\leq n-3$.\\

\noindent\underline{ \textbf{Case $2$ ($t=1$).}} In this case $k=2j+1$. Since the set $\{x, T^{*2}x ,\ldots, T^{*2(2j-n+1)}x\}\subseteq \M$ is linearly independent and $\{e_{0}, e_{1},\ldots, e_{2n-2j-3}\}\subseteq \M$, therefore $\M= \text{span}\{e_{0},\ldots, e_{2n-2j-3},T^{*2(2j-n+1)}x,\ldots, x \}$. Finally substituting the value of $n+p=2j+1$, we get
$$
\boxed{\M= \text{span}\{e_{0},\ldots, e_{n-p-2},T^{*2p}x,\ldots, x \}}~, 
$$
where $\langle x,e_{n+p}\rangle \neq 0$ for some $-1\leq p\leq n-3$.

This completes the proof.
\end{proof}
\begin{remark}\label{rem 2.6}
If we take $T^{*}$ to be the unicellular backward weighted shift where the weights $w_{n}$ may not all equal to 1 and take $\M$ to be a non-cyclic finite-dimensional subspace of $\H$ invariant under $T^{*2}$. Then $\M$ is contained in $M_{k}$ for some 
$k.$ Further, $X\M$ is a non-cyclic finite-dimensional subspace which is invariant under $T_{1}^{*2},$ where $T_{1}^{*}$ is the backward weighted shift on $M_{k}$ with all the weights equal to 1 and $ X:M_{k}\rightarrow{M_{k}}$ with $Xe_{n}= \delta_{n}e_{n} ~ \text{for} ~\delta_{n}=w_{0}w_{1}\cdots w_{n-1}, ~ 1\leq n\leq k~\text{and}~\delta_{0}=1$. Now, according to Theorem \ref{thm 2.5}, there exists a vector $x$ in the space $ X\M$ such that
$$
 X\M= \text{span}~\{e_{0},e_{1}\ldots, e_{n-p-2},~T_{1}^{*2p}x,T_{1}^{*2(p-1)}x\ldots, x \},
$$
where $\langle x,e_{n+p}\rangle \neq 0$ for some $-1\leq p\leq n-3$.
Then we can deduce that 
$$
\M=\text{span}~\{e_{0},e_{1}\ldots,e_{n-p-2},~X^{-1}T_{1}^{*2p}x,X^{-1}T_{1}^{*2(p-1)}x\ldots, X^{-1}x \}.
$$
Therefore  
$$
\boxed{\M=\text{span}~\{e_{0},e_{1}\ldots,e_{n-p-2},~T^{*2p}y,T^{*2(p-1)}y\ldots, y \}}~,
$$ 
where $y=X^{-1}x\in \M$ and $\langle y,e_{n+p}\rangle \neq 0.$ Hence, the general form of non-cyclic finite-dimensional invariant subspaces of $T^{*2}$ remains the same as 
given by Theorem \ref{thm 2.5} irrespective of whether the weights corresponding to $T$ are all equal to 1 or not.  
\end{remark}

\vspace{.3cm}
We now give the general form of a non-cyclic finite-dimensional subspace which is invariant under $T^{*3}.$ Suppose $\M$ is a finite-dimensional subspace of $\H$ that is invariant under $T^{*3}.$ Then, as explained earlier, $\M\subseteq M_{k}$ for some $k$. If dim $\M=n$, then we can easily verify that:

\begin{enumerate}[(i)]
\item $\M$ is non-cyclic if and only if $\M\subseteq M_{k}$ for some $n-1\leq k\leq 3n-4$.
    \item  $\M$ is cyclic if and only if $\M\subseteq M_{k}$ for some $3n-3\leq k\leq 3n-1$.
\end{enumerate}

\begin{theorem}\label{ thm 2.8}
Let $\M$ be a non-cyclic $n$-dimensional subspace of $\H$ invariant under $T^{*3}$. Suppose $p$ is the least integer such that $\M\subseteq M_{n+p}$ and $x\in \M$ with $\langle x,e_{n+p}\rangle \neq 0$. The following cases list all the possible forms that $\M$ can assume: 
\begin{enumerate}
     \item If there exists $y\in \M$ such that $\{T^{*3(\frac{n+p-t}{3})}x, T^{*3(\frac{2n-p-6+t}{3})}y\}$ is linearly independent, then
    \begin{align*}
        \M= \text{span}~\{x,T^{*3}x,\ldots,T^{*3(\frac{n+p-t}{3})}x,y,T^{*3}y\ldots, T^{*3(\frac{2n-p-6+t}{3})}y\},
    \end{align*}
    where $n+p=3j+t$ for some $j$ and $0\leq t\leq 2$.
    \item If there exists $y\in \M$ such that either $\{T^{*3(\frac{n+p-t}{3})}x, T^{*3(\frac{n+p-3-3r}{3})}y\}$ is linearly independent when $n+p=3j$ for some $j$ or  $\{T^{*3(\frac{n+p-t}{3})}x, T^{*3(\frac{n+p-2-3r}{3})}y\}$ is linearly independent when $n+p=3j+2$ for some $j$, then
     $$
      \M= \text{span}~\{e_{0}, e_{1},\ldots, e_{n-2p-3+3r},x, T^{*3}x ,\ldots, T^{*3(p-r)}x,~y, T^{*3}y ,\ldots, T^{*3(p-2r)}y\}
    $$
    for some $0\leq r\leq \frac{p+1}{2}$. Further, if $n+p=3j+1$ for some $j$ and $\{T^{*3(\frac{n+p-1}{3})}x, T^{*3(\frac{n+p-1}{3})}y\}$ is linearly independent, then $r$ must be zero, in which case
    $$
      \M= \text{span}~\{e_{0}, e_{1},\ldots, e_{n-2p-3+},x, T^{*3}x ,\ldots, T^{*3(p)}x, y, T^{*3}y ,\ldots, T^{*3(p)}y\}.
    $$
    \item If there exists $y\in \M$ such that $\{T^{*3(\frac{n+p-1}{3})}x, T^{*3(\frac{n+p-4-3r}{3})}y\}$ is linearly independent when $n+p=3j+1$ for some $j$, then
    $$
     \M= \text{span}~\{e_{0}, e_{1},\ldots, e_{n-2p-2+3r},x, T^{*3}x ,\ldots, T^{*3(p-r)}x, y, T^{*3}y ,\ldots, T^{*3(p-2r-1)}y\}
    $$
     for some $0\leq r\leq \frac{p}{2}$.
\end{enumerate}
\end{theorem}
\begin{proof}
Since $\M$ is a non-cyclic $T^{*3}$-invariant subspace such that dim $\M=n$, then $\M\subseteq M_{n+p}~ (-1\leq p\leq 2n-4)$. 

Let $n+p=3j+t$ for some $j\in \mathbb{N}$ and $0\leq t\leq 2$. Suppose there exist a $x\in \M$ such that
   \begin{equation}\label{eq 2.3}
        x= \sum_{i=0}^{3j+t}\alpha_{i}e_{i},\quad\alpha_{3j+t}\neq 0.
   \end{equation}
   
 Then according to Lemma \ref{lem 2.1}, the set $\langle x\rangle_{T^{*3}}= \text{span}~\big\{x, T^{*3}x ,\ldots, T^{*3j}x\big\}$ is linearly independent and hence dim $\langle x\rangle_{T^{*3}}=j+1$.
   
 Since $\M$ is non-cyclic, then according to Theorem \ref{thm 2.2} there exist a subspace $\W$ of $\M$ invariant under $T^{*3}$ such that $\M=\langle x\rangle_{T^{*3}} \oplus \W$. Now, $\text{dim}~ \W= \text{dim}~ \M- \text{dim}~\langle x\rangle_{T^{*3}}=n-j-1$. This yields that $ \W\subseteq M_{3n-3j-4}$, otherwise, if there exist an $v\in \W$ such that $\langle v,e_{i}\rangle \neq 0$ for some $i\geq 3n-3j-3$, then the set $\{y, T^{*2}y ,\ldots, T^{*3(n-j-1)}y\} \subseteq \W$ implying that dim $\W\geq n-j$.

Note that, according to Corollary \ref{cor 2.3} the subspace $\M$ can be decomposed into at most three $T^{*3}$-invariant cyclic subspaces. Then either there exist a $y\in \W$ such that $T^{*3(n-j-2)}y\neq 0$ or $T^{*3(n-j-2)}y=0$ for all $y\in \W$.

Let us begin with the first possibility. Suppose there exist a $y\in \W$ such that $T^{*3(n-j-2)}y\neq 0$. In this situation, $ \M= \langle x\rangle_{T^{*3}} \oplus \langle y\rangle_{T^{*3}}$, where $ W=\langle y\rangle_{T^{*3}}= \text{span}~\{y, T^{*3}y ,\ldots, T^{*3(n-j-2)}y\}$. Then
 $$
  \M= \text{span}~\{x, T^{*3}x ,\ldots, T^{*3j}x,~y, T^{*3}y ,\ldots, T^{*3(n-j-2)}y\}.
 $$
  We can express the above form in another way
$$
 \boxed{ \M= \text{span}~\{x,T^{*3}x,\ldots,T^{*3(\frac{n+p-t}{3})}x,~y,T^{*3}y\ldots, T^{*3(\frac{2n-p-6-t}{3})}y \}}~,
$$
 where $x\in M_{n+p}$ such that $\langle x,e_{n+p}\rangle \neq 0$ and $n+p=3j+t$ for some $j\in \mathbb{N},0\leq t\leq 2$.\\
 
For the second possibility, assume that $T^{*3(n-j-2)}y=0$ for all $y\in \W$. Then $ \M\neq \langle x\rangle_{T^{*3}}  \oplus \W$. Going forward, to present the proof into a convenient way, we have divided it into three cases.\\

\noindent\underline{\textbf{Case $1$ ($n+p=3j$).}} Recall, from Equation (\ref{eq 2.3}) there exist an $x\in \M$ such that
   \begin{equation}\label{eq 2.4}
        x= \sum_{i=0}^{3j}\alpha_{i}e_{i},\quad\alpha_{3j}\neq 0\quad\text{and}\quad \text{dim}~\langle x\rangle_{T^{*3}}=j+1.
   \end{equation}
 
First of all if there exist $v\in \W$ such that $v= \sum_{i=0}^{3j}v_{i}e_{i},~v_{3j}\neq 0$, then the set $\{T^{*3j}x, T^{*3j}v\}\subseteq M_{0}$, which cannot be linearly independent, therefore $T^{*3j}v=0$ for all $v\in \W$. Suppose there exist a $y\in \W$ such that $T^{*3(j-1)}y\neq 0$. Then
  \begin{equation}\label{eq 2.5}
        y= \sum_{i=0}^{3j-1}\beta_{i}e_{i},\quad\text{where at least one of the coefficient}~\beta_{3j-1}, \beta_{3j-2}\neq 0.
  \end{equation}
  
In particular, $\langle y\rangle_{T^{*3}}= \text{span}~\{y, T^{*3}y ,\ldots, T^{*3(j-1)}y\}$ and dim $\langle y\rangle_{T^{*3}}=j$. Since $ \M\neq \langle x\rangle_{T^{*3}}  \oplus \W$, then according to Corollary (\ref{cor 2.3}) there exist a $z\in \W$ such that $\W= \langle y\rangle_{T^{*3}} \oplus \langle z\rangle_{T^{*3}}$. Therefore
$$
 \text{dim} \langle z\rangle_{T^{*3}}= \text{dim}~\W-\text{dim}~\langle y\rangle_{T^{*3}}=(n-j-1)-j=n-2j-1.
$$

So we have $\langle z\rangle_{T^{*3}}= \text{span}~\{z, T^{*3}z ,\ldots, T^{*3(n-2j-2)}z\}$ and this indicates that $T^{*3(n-2j-1)}z=0$, which guarantees $z\in M_{3n-6j-4} $. In particular
   \begin{equation}\label{eq 2.6}
        z= \sum_{i=0}^{3n-6j-4}\gamma_{i}e_{i},\quad\text{where at least one of the coefficients}~\gamma_{3n-6j-4}, \gamma_{3n-6j-5}\neq 0.
   \end{equation}
   
Equations (\ref{eq 2.4}), (\ref{eq 2.5}) and (\ref{eq 2.6}) implies that 
$$
\M=\text{span}~\{x,\ldots, T^{*3j}x,y,\ldots, T^{*3(j-1)}y,z,\ldots, T^{*3(n-2j-2)}z\}.
$$

Now let us assume that $T^{*3(j-1)}v=0$ for all $v\in \W$. Suppose there exist a $y\in \W$ such that $T^{*3(j-2)}y\neq 0$. Then using the similar arguments as above, we can deduce that there exist a $z\in \W$ such that
 \begin{align*}
 \langle y\rangle_{T^{*3}} &= \text{span}~\{y, T^{*3}y ,\ldots, T^{*3(j-2)}y\}~,~~\text{dim}~\langle y\rangle_{T^{*3}}=j-1,\\
 \langle z\rangle_{T^{*3}} &= \text{span}~\{z, T^{*3}z ,\ldots, T^{*3(n-2j-1)}z\}~,~~\text{and~ dim}~\langle z\rangle_{T^{*3}}=n-2j,
\end{align*}
where $y\in M_{3j-4}$ and $z\in M_{3n-6j-1}$. Then 
$$
\M=\text{span}~\{x,\ldots, T^{*3j}x,y,\ldots, T^{*3(j-2)}y,z,\ldots, T^{*3(n-2j-1)}z\}.
$$

So continuing this way, after $r$-th such steps let us assume $T^{*3(j-r)}v=0$ for all $v\in \W$. Suppose there exist a $y\in \W$ such that $T^{*3(j-1-r)}y\neq 0$. Then there exist a $z\in \W$ such that
    \begin{align}
       \langle y\rangle_{T^{*3}} &= \text{span}~\{y, T^{*3}y ,\ldots, T^{*3(j-1-r)}y\}~,~~\text{ dim}~\langle y\rangle_{T^{*3}}=j-r\label{eq 2.7},\\
       \langle z\rangle_{T^{*3}} &= \text{span}~\{z, T^{*3}z ,\ldots, T^{*3(n-2j-2+r)}z\}~,~~\text{and~ dim}~\langle z\rangle_{T^{*3}}=n-2j-1+r, \label{eq 2.8}
    \end{align}
where $y\in M_{3j-1-3r}$ and $z\in M_{3n-6j-4+3r}$. So after combining Equations (\ref{eq 2.4}), (\ref{eq 2.7}) and (\ref{eq 2.8}) at $r$-th step, we get
$$
\M = \text{span}~\{x, T^{*3}x ,\ldots, T^{*3j}x, y, T^{*3}y ,\ldots, T^{*3(j-1-r)}y, z, T^{*3}z ,\ldots, T^{*3(n-2j-2+r)}z\}.
$$

Since dim $\langle y\rangle_{T^{*3}}\geq$ dim $\langle z\rangle_{T^{*3}}$, therefore $j-r\geq n-2j-1+r$ implying that $0\leq r\leq (3j-n+1)/2$. Furthermore, let us write $\M=X_{1}\oplus X_{2}$, where
\begin{align*}
    X_{1} &=\text{span}~\{x, T^{*3}x ,\ldots, T^{*3(3j-n-r)}x, y, T^{*3}y ,\ldots, T^{*3(3j-n-2r)}y\}\\
\text{and}\quad X_{2} &=\text{span}~\{T^{*3(3j-n-r+1)}x ,\ldots,
T^{*3j}x, T^{*3(3j-n-2r+1)}y ,\ldots, T^{*3(j-1-r)}y, z,\ldots, T^{*3(n-2j-2+r)}z\}.
\end{align*}

Now we will show that $X_{2}=\{e_{0},\ldots, e_{3n-6j-3+3r}\}$. Firstly $T^{*3j}x\in M_{0}$, implying that $e_{0}\in \M$. For $0\leq m\leq n-2j-2+r$, the set $\{T^{*3(j-1-m)}x, T^{*3(j-1-r-m)}y,  T^{*3(n-2j-2+r-m)}z\}\subseteq M_{3m+3}$. Now let us apply induction on $m$.

\vspace{.2cm}
 For $m=0$, the set $\{T^{*3(j-1)}x, T^{*3(j-1-r)}y,  T^{*3(n-2j-2+r)}z\}\subseteq M_{3}$ such that
         \begin{align*}
       x_{1}&= T^{*3(j-1)}x -\alpha_{0}e_{0}= \sum_{i=1}^{3}\alpha_{i}e_{i},\\
           y_{1}&= T^{*3(j-1-r)}y -\beta_{0}e_{0}= \sum_{i=1}^{2}\beta_{i}e_{i},\\
       \text{and}\quad    z_{1}&= T^{*3(n-2j-2+r)}z -\gamma_{0}e_{0}= \sum_{i=1}^{2}\gamma_{i}e_{i}.
        \end{align*}
        
For $a,b,c\in \mathbb{C}$, consider that $a x_{1}+b y_{1}+cz_{1}=0$. Then $T^{*3}(a x_{1}+by_{1}+cz_{1})=0$. So $a\alpha_{3}e_{0}=0$ implying that $a=0$. Now,  
$bT^{*3(j-1-r)}y+cT^{*3(n-2j-2+r)}z=(b\beta_{0}+c\gamma_{0})e_{0}\in \text{span}~\{T^{*3(j)}x\}$. This is a contradiction to the fact that the set $\{T^{*3j}x, T^{*3(j-1-r)}y,  T^{*3(n-2j-2+r)}z\}$ is linearly independent. Therefore the set $\{x_{1},y_{1},z_{1}\}\subseteq M_{3}\setminus M_{0}$ is linearly independent and solving above equations gives us that $e_{1}, e_{2}, e_{3}\in \M$.

\vspace{.2cm}
 For $m=l$, let us assume that $\{e_{1}, e_{2},\ldots, e_{3l+3}\}\subseteq \M$.

\vspace{.2cm}
 For $m=l+1$, we have $\{T^{*3(j-l-2)}x, T^{*3(j-r-l-2)}y,  T^{*3(n-2j+r-l-3)}z\}\subseteq M_{3l+6}$. Then using Equations (\ref{eq 2.4}), (\ref{eq 2.7}) and (\ref{eq 2.8}) we get
\begin{align*}
    T^{*3(j-l-2)}x = \sum_{i=0}^{3l+6}\alpha_{i}e_{i},\quad
            T^{*3(j-r-l-2)}y = \sum_{i=0}^{3l+6}\beta_{i}e_{i}\quad\text{and}\quad
            T^{*3(n-2j+r-l-3)}z = \sum_{i=0}^{3l+6}\gamma_{i}e_{i}.
\end{align*}
Let us consider the following
        \begin{align*}
            x_{1}&=T^{*3(j-l-2)}x-\sum_{i=0}^{3l+3}\alpha_{i}e_{i}\in M_{3l+6}\setminus M_{3l+3}, \\
           y_{1}&=T^{*3(j-r-l-2)}y - \sum_{i=0}^{3l+3}\beta_{i}e_{i}\in M_{3l+6}\setminus M_{3l+3},\\
           \text{and}\quad z_{1}&= T^{*3(n-2j+r-l-3)}z-\sum_{i=0}^{3l+3}\gamma_{i}e_{i}\in M_{3l+6}\setminus M_{3l+3}.
        \end{align*}
Using the cases for $m=0,l$ and by similar arguments as above, we can deduce that the set $\{x_{1}, y_{1}, z_{1} \}$ is linearly independent in span $\{e_{3l+4}, e_{3l+5}, e_{3l+6}\}$. Then $e_{3l+4}, e_{3l+5}, e_{3l+6}\in \M$. Hence by induction the set $X_{2}=\{e_{0}, e_{1},\ldots, e_{3n-6j-3+3r}\}\subseteq \M$.

Moreover, the  set $X_{1}=\{x, T^{*3}x ,\ldots, T^{*3(3j-n-r)}x,~ y, T^{*3}y ,\ldots, T^{*3(3j-n-2r)}y\}\subseteq \M$ is linearly independent and dim $X_{1}=6j-2n-3r+2$, therefore
$$
   \M= \text{span}~\{e_{0}, e_{1},\ldots, e_{3n-6j-3+3r},~x, T^{*3}x ,\ldots, T^{*3(3j-n-r)}x,~ y, T^{*3}y ,\ldots, T^{*3(3j-n-2r)}y\}.
$$

Lastly, substituting the value of $n+p=3j$ in above form mentioned, we get
     \begin{align*}
       \boxed{ \M= \text{span}~\big\{e_{0}, e_{1},\ldots, e_{n-2p-3+3r},~x, T^{*3}x ,\ldots, T^{*3(p-r)}x,~ y, T^{*3}y ,\ldots, T^{*3(p-2r)}y\big\}}~.
    \end{align*}
 
 \vspace{.3cm}
\noindent\underline{\textbf{Case $2$ ($n+p=3j+1$).}} Recall, from Equation (\ref{eq 2.3}) there exist an $x\in \M$ such that
  \begin{equation}\label{eq 2.9}
       x= \sum_{i=0}^{3j+1}\alpha_{i}e_{i},~\alpha_{3j+1}\neq 0\quad\text{and\quad dim}~\langle x\rangle_{T^{*3}}=j+1.
  \end{equation}
 
We will further divide this case into two subcases depending on whether there exists an $v\in \W$ such that $T^{*3j}v\neq 0$ or not.

\noindent\underline{\textbf{Subcase $1$.}} Suppose there exist a $y\in \W$ such that $T^{*3j}y\neq 0$. Then
\begin{equation}\label{eq 2.10}
    y= \sum_{i=0}^{3j+1}\beta_{i}e_{i},\quad\text{where at least one of the coefficient's}~\beta_{3j}, \beta_{3j+1}\neq 0.
\end{equation}

In particular, $\langle y\rangle_{T^{*3}}= \text{span}~\{y, T^{*3}y ,\ldots, T^{*3j}y\}$ and dim $\langle y\rangle_{T^{*3}}=j+1$. Now proceeding in the same manner as Case $1$, we get that there exists a $z\in \W$ such that $\W=\langle y\rangle_{T^{*3}}\oplus \langle z\rangle_{T^{*3}}$. Thus
$$
\text{dim}~ \langle z\rangle_{T^{*3}}= \text{dim}~\W-\text{dim}~\langle y\rangle_{T^{*3}}=(n-j-1)-(j+1)=n-2j-2.
$$

So we have $\langle z\rangle_{T^{*3}}= \text{span}~\{z, T^{*3}z ,\ldots, T^{*3(n-2j-3)}z\}$. Then $T^{*3(n-2j-2)}z=0$, so $z\in M_{3n-6j-7} $. Therefore
\begin{equation} \label{eq 2.11}
    z= \sum_{i=0}^{3n-6j-7}\gamma_{i}e_{i}~\text{ where}~\gamma_{3n-6j-7}\neq 0.
\end{equation}

Then by Equations (\ref{eq 2.9}), (\ref{eq 2.10}) and (\ref{eq 2.11}) we get
$$
M= \text{span}~\{x, T^{*3}x ,\ldots, T^{*3j}x,~ y, T^{*3}y ,\ldots, T^{*3j}y, ~z, T^{*3}z ,\ldots, T^{*3(n-2j-3)}z\}.
$$

Furthermore, let us write $\M=X_{1}\oplus X_{2}$, where
\begin{align*}
    X_{1} &=\text{span}~\{x, T^{*3}x ,\ldots, T^{*3(3j-n+2)}x, y, T^{*3}y ,\ldots, T^{*3(3j-n+2)}y\}\\
   \text{and}\quad X_{2} &=\text{span}~\{T^{*3(3j-n+3)}x ,\ldots,
T^{*3j}x, T^{*3(3j-n+3)}y ,\ldots, T^{*3j}y, z,\ldots, T^{*3(n-2j-3)}z\}.
\end{align*}

For $0\leq m\leq n-2j-3$, the set $\{T^{*3(j-m)}x, T^{*3(j-m)}y,  T^{3(n-2j-3-m)}z\}\subseteq M_{3m+2}$. Applying induction on $\M$ and using the similar arguments as Case $1$, we can deduce that $ X_{2}=\{e_{0},\ldots, e_{3n-6j-7}\}\subseteq \M$.

Moreover
$$
 T^{*3(3j-n+2)}x =\sum_{i=0}^{3n-6j-5}\alpha_{i}e_{i}~~\text{and}~~T^{*3(3j-n+2)}y = \sum_{i=0}^{3n-6j-5}\beta_{i}e_{i}
$$
are linearly independent in the set $X_{1}$. Since $\{e_{0}, e_{1},\ldots, e_{3n-6j-7}\}\subseteq \M$, therefore using the similar arguments as in Case $1$ it is easy to deduce that $e_{3n-6j-6}, e_{3n-6j-5}\in \M$.

Also, the set $\{x, T^{*3}x ,\ldots, T^{*3(3j-n+1)}x,~ y, T^{*3}y ,\ldots, T^{*3(3j-n+1)}y\}\subseteq X_{1}$ is linearly independent and it has dimension equal to \say{$6j-2n+4$}. Then
$$
 \M= \text{span}~\{e_{0}, e_{1},\ldots, e_{3n-6j-5},~x, T^{*3}x ,\ldots, T^{*3(3j-n+1)}x,~ y, T^{*3}y ,\ldots, T^{*3(3j-n+1)}y\}.
$$

Lastly, substituting the value of $n+p=3j+1$ in the above mentioned form of $\M$, we get
$$
 \boxed{ \M= \text{span}~\{e_{0}, e_{1},\ldots, e_{n-2p-3},~x, T^{*3}x ,\ldots, T^{*3p}x,~ y, T^{*3}y ,\ldots, T^{*3p}y\}}~.
$$

\noindent\underline{\textbf{Subcase $2$.}} Let $T^{*3j}v=0$ for all $v\in \W$. Suppose there exist a $y\in \M$ such that
 \begin{equation}\label{eq 2.12}
      y= \sum_{i=0}^{3j-1}\beta_{i}e_{i},\quad\text{where at least one of the coefficient's}~\beta_{3j-3},\beta_{3j-2} \beta_{3j-1}\neq 0.
 \end{equation} 
 
 In particular, $\langle y\rangle_{T^{*3}}= \text{span}~\{y, T^{*3}y ,\ldots, T^{*3(j-1)}y\}$ and dim $\langle y\rangle_{T^{*3}}=j$. Proceeding as Case $1$, we get that there exists a $z\in \W$ such that $\W=\langle y\rangle_{T^{*3}}\oplus \langle z\rangle_{T^{*3}}$. Thus
$$
 \text{dim} ~\langle z\rangle_{T^{*3}}= \text{dim}~  \W-\text{dim}~\langle y\rangle_{T^{*3}}=(n-j-1)-j=n-2j-1.
$$

We have $\langle z\rangle_{T^{*3}}= \text{span}~\{z, T^{*3}z ,\ldots, T^{*3(n-2j-2)}z\}$. Then $T^{*3(n-2j-1)}z=0$, implying that $z\in M_{3n-6j-4} $. Therefore 
    \begin{equation}\label{eq 2.13}
        z= \sum_{i=0}^{3n-6j-4}\gamma_{i}e_{i},\quad\gamma_{3n-6j-4}\neq 0.
    \end{equation}
    
 So continuing this way, at the $r$-th step let us assume $T^{*3(j-r)}v=0$ for all $v\in \W$. Again using the similar arguments as Case $1$ we can prove there exist $y,z\in \W$ such that 
    \begin{align}
      \langle y\rangle_{T^{*3}} &= \text{span}~\{y, T^{*3}y ,\ldots, T^{*3(j-1-r)}y\}~,~~\text{ dim}~\langle y\rangle_{T^{*3}}=j-r,\label{eq 2.14}\\
       \langle z\rangle_{T^{*3}} &= \text{span}~\{z, T^{*3}z ,\ldots, T^{*3(n-2j-2+r)}z\}~,~~\text{and~ dim}~\langle z\rangle_{T^{*3}}=n-2j-1+r.\label{eq 2.15} 
    \end{align}
where $y\in M_{3j-1-3r}$ and $z\in M_{3n-6j-4+3r}$. Now combining Equations (\ref{eq 2.9}), (\ref{eq 2.14}) and (\ref{eq 2.15}) at the $r$-th step, we get
$$
 \M= \text{span}~\{x, T^{*3}x ,\ldots, T^{*3j}x,~ y, T^{*3}y ,\ldots, T^{*3(j-1-r)}y, ~z, T^{*3}z ,\ldots, T^{*3(n-2j-2+r)}z\}.
$$

Since dim $\langle y\rangle_{T^{*3}}\geq$ dim $\langle z\rangle_{T^{*3}}$, then $j-r\geq n-2j-1+r$ implying that $ 0\leq r\leq (3j-n+1)/2$. Furthermore, let us write $\M=X_{1}\oplus X_{2}$, where
\begin{align*}
     X_{1} &=\text{span}~\{x, T^{*3}x ,\ldots, T^{*3(3j-n+1-r)}x, y, T^{*3}y ,\ldots, T^{*3(3j-n-2r)}y\}\\
    \text{and}\quad X_{2} &=\text{span}~\{T^{*3(3j-n+2-r)}x ,\ldots,
T^{*3j}x, T^{*3(3j-n-2r+1)}y ,\ldots, T^{*3(j-1-r)}y, z,\ldots, T^{*3(n-2j-2+r)}z\}.
\end{align*}

Applying similar arguments as in Case $1$ we can prove that $X_{2}= \{e_{0},\ldots, e_{3n-6j-4+3r}\}\subseteq \M$. Also, we have dim $X_{1}=6j-2n+3-3r$. Then
$$
 \M= \text{span}~\{e_{0}, e_{1},\ldots, e_{3n-6j-4+3r},~x, T^{*3}x ,\ldots, T^{*3(3j-n+1-r)}x,~ y, T^{*3}y ,\ldots, T^{*3(3j-n-2r)}y\}.
$$

Lastly, substituting the value of $n+p=3j+1$ in the above mentioned form of $\M$, we get
$$
\boxed{ \M= \text{span}~\{e_{0}, e_{1},\ldots, e_{3n-6j-4+3r},~x, T^{*3}x ,\ldots, T^{*3(p-r)}x,~ y, T^{*3}y ,\ldots, T^{*3(p-1-2r)}y\}}~.
$$
 
 \vspace{.3cm}
\noindent\underline{\textbf{Case $3$ ($n+p=3j+2$).}} Recall, from Equation (\ref{eq 2.3}) there exist an $x\in \M$ such that
  \begin{equation}\label{eq 2.16}
       x= \sum_{i=0}^{3j+2}\alpha_{i}e_{i},\quad\alpha_{3j+2}\neq 0\quad\text{and\quad dim}~\langle x\rangle_{T^{*3}}=j+1.
  \end{equation}
    
Following the same path as in Case $1$, at the $r$-th step let us assume that $T^{*3(j+1-r)}v=0$ for all $v\in \W$. Again we can prove there exist $y,z\in \W$ such that 
    \begin{align}
      \langle y\rangle_{T^{*3}} &= \text{span}~\{y, T^{*3}y ,\ldots, T^{*3(j-r)}y\}~,~~\text{ dim}~\langle y\rangle_{T^{*3}}=j-r+1,\label{eq 2.17}\\
      \langle z\rangle_{T^{*3}} &= \text{span}~\{z, T^{*3}z ,\ldots, T^{*3(n-2j-3+r)}z\}~,~~\text{and~ dim}~\langle z\rangle_{T^{*3}}=n-2j-2+r,\label{eq 2.18} 
    \end{align}
    where $y\in M_{3j+2-3r}$ and $z\in M_{3n-6j-7+3r}$. Now combining Equations (\ref{eq 2.16}), (\ref{eq 2.17}) and (\ref{eq 2.18}) at the $r$-th step, we get
$$
  \M= \text{span}~\{x, T^{*3}x ,\ldots, T^{*3j}x,~ y, T^{*3}y ,\ldots, T^{*3(j-r)}y, ~z, T^{*3}z ,\ldots, T^{*3(n-2j-3+r)}z\}.
$$

Since dim $\langle y\rangle_{T^{*3}}\geq$ dim $\langle z\rangle_{T^{*3}}$, then $j+1-r\geq n-2j-2+r$ implying that $ 0\leq r\leq (3j-n+3)/2$. Repeating the process as in Case $1$, we can deduce that
$$
\boxed{M= \text{span}~\{e_{0}, e_{1},\ldots, e_{n-2p-3+3r},~x, ~T^{*3}x ,\ldots, T^{*3(p-r)}x,~ y, ~T^{*3}y ,\ldots, T^{*3(p-2r)}y\}}~.
$$
This completes the proof.
\end{proof}

\begin{remark}\label{rem 2.9}
By using the similar set of arguments as were used in Remark \ref{rem 2.6}, we deduce that the general forms that a non-cyclic finite-dimensional subspace invariant under $T^{*3}$ can assume remains the same as are given by Theorem \ref{ thm 2.8} irrespective of whether all weights corresponding to $T$ equals 1 or not.
\end{remark}

\begin{remark}
 If $\M$ is a finite-dimensional cyclic subspace of $\H$ that is invariant under $T^{*i}$ for $2\le i\le 3$, then all we 
can say is that $\M=\text{span}~\{x, T^{*i}x, T^{*2i}x, \dots, T^{*(n-1)i}x\}$ for some $n$ and $x\in \M.$ In this case, unlike the case of 
non-cyclic subspace, we can't guarantee the existence of any $e_i$ in $\M$. For example $\M=\text{span}~\{e_{0}+e_{1},e_{2}+e_{3},e_{4}+e_{5}\}$ is a cyclic 3-dimensional $T^{*2}$-invariant subspace and  $\M=\text{span}~\{e_{0}+e_{1},e_{3}+e_{4},e_{6}+e_{7}\}$ is a 
cyclic 3-dimensional $T^{*3}$-invariant subspace, but neither of them contains any $e_i.$
\end{remark}

\vspace{.3cm}
Using Theorem \ref{ thm 2.8}, we know that in order to construct a concrete example of a finite-dimensional non-cyclic 
$T^{*3}$- invariant subspace, one must start with two linearly independent vectors $x$ and $y$ so that the entire set $\langle x\rangle_{T^{*l}}\cup \langle y\rangle_{T^{*l}}$ is also linearly independent. In general, this can be a very tedious task to carry out. Our next result simplify this work by giving a necessary and sufficient condition for the set $\langle x\rangle_{T^{*l}}\cup \langle y\rangle_{T^{*l}}$ to be linearly independent.

\begin{pro}\label{pro 2.9}
  Let $x,y\in \H$ and $\langle x\rangle_{T^{*l}}=\{x, T^{*l}x ,\ldots, T^{*lr}x\}$ and $\langle y\rangle_{T^{*l}}=\{y, T^{*l}y ,\ldots, T^{*ls}y\}$ for some fixed $l,r,s\in \mathbb{N}$. Then $\langle x\rangle_{T^{*l}}\cup \langle y\rangle_{T^{*l}}$ is linearly independent if and only if the set \big\{$T^{*lr}x,T^{*ls}y\big\}$ is linearly independent.
\end{pro}
\begin{proof}
Let $x,y\in M_{k}$ for some $k\in \mathbb{N}$ such that $x=\sum_{j=0}^{n}\alpha_{j}e_{j}~\text{and}~y=\sum_{j=0}^{m}\beta_{j}e_{j}$ and $\alpha_{n},\beta_{m}\neq 0$. Also for some fixed $l,s,r\geq 0$ we have $T^{*lr}x\neq 0,T^{*ls}y\neq0$ and $T^{*l(r+1)}x=0,T^{*l(s+1)}y=0,~lr\leq n$ and $ls\leq m $. Let the set $\langle x\rangle_{T^{*l}}\cup \langle y\rangle_{T^{*l}}$ is linearly independent. Then automatically the set \{$T^{*lr}x,T^{*ls}y\}$ is also linearly independent.

Conversely, suppose $T^{*lr}x$ and $T^{*ls}y$ are linearly independent. First assume that $n\leq m$, then $r\leq s$. Now we can have the following cases;\\

\noindent\underline{\textbf{Case $1$ ($r=s$).}} Then the set \{$T^{*lr}x,T^{*lr}y\}$ is linearly independent. Now consider
    $$
     \sum_{j=0}^{r}\gamma_{j}T^{*jl}x +  \sum_{j=0}^{r}\delta_{j}T^{*jl}y =0.
    $$
    
Applying $T^{*lr}$ on both sides of the above equation we get $\gamma_{0}T^{*lr}x+ \delta_{0}T^{*lr}y=0$. Then $\gamma_{0}, \delta_{0}=0$ which implies that
$$
     \sum_{j=1}^{r}\gamma_{j}T^{*jl}x +  \sum_{j=1}^{r}\delta_{j}T^{*jl}y =0.
$$

Similarly, applying $T^{*l(r-1)}$ on both side of the above equation, we get $\gamma_{1}T^{*lr}x+ \delta_{1}T^{*lr}y=0$ which implies that $\gamma_{1}, \delta_{1}=0$. 

Continuing in the similar fashion and applying $T^{*l(r-i)}$ on both sides of the equation for $2\leq i\leq r$, we have that $\gamma_{0},\cdots,\gamma_{r}, \delta_{0},\cdots,\delta_{r}=0$. Hence the set $\langle x\rangle_{T^{*l}}\cup \langle y\rangle_{T^{*l}}$ is linearly independent.\\

\noindent\underline{\textbf{Case $2$ ($r<s$).}} Let us consider
$$
 \sum_{j=0}^{r}\gamma_{j}T^{*jl}x +  \sum_{j=0}^{s}\delta_{j}T^{*jl}y =0.
$$

Applying $T^{*ls}$ on both side of the above equation we get, $\delta_{0} T^{*ls}y= 0$, which gives $\delta_{0}=0$. Thus
$$
 \sum_{j=0}^{r}\gamma_{j}T^{*jl}x +  \sum_{j=1}^{s}\delta_{j}T^{*jl}y =0.
$$

If $r<s-1$, then applying $T^{*l(s-1)}$ on both side of the above equation we get, $\delta_{1} T^{*ls}y= 0$, then $\delta_{1}=0$. Continuing this way and applying $T^{*l(s-i)}$ on both sides of the equation for $2\leq i\leq s-r-1$, we have that $\delta_{0},\delta_{1},\cdots ,\delta_{s-r-1}=0$. Thus
$$
 \sum_{j=0}^{r}\gamma_{j}T^{*jl}x +  \sum_{j=s-r}^{s}\delta_{j}T^{*jl}y =0.
$$

Applying $T^{*lr}$ on both sides we get $\gamma_{0} T^{*lr}x + \delta_{s-r}T^{*ls}y= 0$, which gives $\gamma_{0},\delta_{s-r}=0$.

Again $T^{*l(r-1)}( \sum_{j=1}^{r}\gamma_{j}T^{*jl}x +  \sum_{j=s-r+1}^{s}\delta_{j}T^{*jl}y) =0$. Then $\gamma_{1} T^{*lr}x + \delta_{s-r+1}T^{*ls}y= 0$ implying that $\gamma_{1},\delta_{s-r+1}=0$.

Continuing in the same manner we can prove  $\gamma_{0},\cdots,\gamma_{r-1} ,\delta_{s-r},\cdots,\delta_{s-1}=0$ and eventually we get $\gamma_{r}T^{*lr}x + \delta_{s}T^{*ls}y= 0,~\text{which implies that}~ \gamma_{r},\delta_{s}=0$.
Hence the set $\langle x\rangle_{T^{*l}}\cup \langle y\rangle_{T^{*l}}$ is linearly independent. This completes the proof.
\end{proof}

Our final result of this section characterizes finite-dimensional subspaces of $\H$ that are jointly invariant under $T^{*2}$ and $T^{*3}.$ 

\begin{theorem}\label{thm 2.10}
Let $\M\subseteq \H$ be a closed subspace such that dim $\M=n$ and $\M$ is jointly invariant under $T^{*2}$ and $T^{*3}.$ Then $\M=\text{span}~\{e_{0},e_{1},\ldots, e_{n-2},\alpha e_{n-1}+\beta e_{n}\}$ for a non-zero pair $(\alpha,\beta)\in \mathbb{C}^{2}$. 
\end{theorem}
\begin{proof}
First of all we will prove that $\M\subseteq M_{k}$ for $n-1\leq k\leq n$. Since dim $\M=n$ and dim $M_{n-2}=n-1$, then clearly $k\geq n-1$.
Suppose there exist a $x\in \M$ such that $\langle x,e_{n+1}\rangle \neq 0$ and $x=\sum_{i=0}^{n+1}x_{i}e_{i}$ for some $x_{i}\in \mathbb{C}$. As $\M$ is invariant under $T^{*2}$ and $T^{*3}$, therefore it is easy to see that $T^{*j}\M\subseteq \M$ 
for all $j\geq 2$. Now according to Lemma \ref{lem 2.1} the set $\W=\{x, T^{*2}x,T^{*3}x,\ldots,T^{*n+1}x\}\subseteq \M$ is linearly independent, which is a contradiction to the fact that dim $\W\leq\text{dim}~\M$. Therefore, $\M\subseteq M_{k}$ for $n-1\leq k\leq n$. 
 
Let $y=\sum_{i=0}^{n}y_{i}e_{i}$ be a element of $\M$ such that at least one of the coefficients $y_{n-1},y_{n}\neq 0$. If $n=1$, then clearly $\M=\text{span}~\{y_{0} e_{0}+y_{1} e_{1}\}$ for any non-zero pair of complex numbers $(y_{0},y_{1})\in \mathbb{C}^{2}$. Suppose $y_{n}\neq 0$, so if $n\geq 2$ then $T^{*n}y=w_{0}w_{1}\cdots w_{n-1}y_{n}e_{0}$, which means that $e_{0}\in \M$. Again $T^{*n-1}y=w_{0}w_{1}\cdots w_{n-2}y_{n-1}e_{0}+w_{1}w_{2}\cdots w_{n-1}y_{n}e_{1}$. Since $e_{0}\in \M$, therefore $e_{1}\in \M$. By continuing this procedure, it is easy to show that at each step $T^{*n-r}y\in \M$ implies that $e_{j}\in \M$ for $0\leq r\leq n-2$. Therefore
$$
\M=\text{span}~\{e_{0},e_{1},\ldots, e_{n-2},\alpha e_{n-1}+\beta e_{n}\}\quad\text{such that}\quad(\alpha,\beta)\neq (0,0).
$$
This completes the proof.
\end{proof}


\section{Infinite-dimensional subspaces}\label{sec 3}
In this section we will characterize infinite-dimensional subspaces that are invariant under $T^{*l}$ for $2\le l\le 3$ and infinite-dimensional subspaces that are jointly invariant under $T^{*2}$ and $T^{*3}$ for two classes of weights. The invariant subspaces for these weights were characterized by Nikolskii in \cite{nikolskii1965invariant} and Yadav \& Chaterjee in \cite{yadav1982characterization}. First we show that the condition of bounded variation on the weights considered by the authors in \cite{yadav1982characterization} is redundant. To state their result and establish our claim, we start with some definitions and remarks.

\begin{defn}
A set of vectors 
$\{x_{0},x_{1},\ldots,x_{n},\ldots\}\subseteq \H$ is called \textit{$\omega$-independent} if $\sum_{k=0}^{\infty}\alpha_{k}x_{k}=0$ for $\alpha_{k}\in \mathbb{C}$ implies that $\alpha_{k}=0$ for all $k\geq 0, k\in \mathbb{N}$.
\end{defn}

\begin{remark}\label{rem 3.2}
 Let $T$ be the forward weighted shift with weight sequence $\{w_{n}\}$ of bounded positive real numbers, then for any $x\in \H$ the set $\{x,Tx,T^{2}x,\ldots \}$ is $\omega$-independent.
\end{remark}

\begin{defn}
 Two sequence of vectors $\{f_{n}\}$ and $\{g_{n}\}$ in a normed vector space are said to be \textit{quadratically close} if $\sum_{n=0}^{\infty} \|f_{n}-g_{n}\|^{2}< \infty$.
\end{defn}
\begin{defn}
 A bounded sequence $\{w_{n}\}_{n \in \mathbb{N}}$ of positive real numbers is said to be of bounded variation if $\sum_{n=0}^{\infty}|w_{n}-w_{n+1}|<\infty$.
\end{defn} 
   
\begin{theorem}[\cite{yadav1982characterization}, Theorem 1]\label{thm 3.5}
Let $\{w_{n}\}_{n \in \mathbb{N}}$ be of bounded variation such that 
\begin{align}\label{eq  3.1}
    \delta= \underset{m\geq2,n}{\sup}~\sum_{k=0}^{\infty}\bigg(\frac{w_{k+m}\cdots w_{k+n}}{w_{m}\cdots w_{n}}\bigg)^{2}<\infty.
\end{align}
Then $\text{Lat}~T=\big\{\{0\},L_{1},L_{2},\ldots,L_{n},\cdots, \H\big\}$ where $ L_{n} = \bigvee_{i=n}^{\infty}\{e_{i}\}$.
\end{theorem}

In the following result we obtain the same characterization for invariant subspaces for $T$ as Theorem \ref{thm 3.5} but without assuming the condition of bounded variation on the weights.

 \begin{theorem}\label{thm 3.6}
Let $T :\H\rightarrow{\H}$ be a forward weighted shift with weights sequence $\{w_{n}\}_{n \in \mathbb{N}}$ satisfying Condition (\ref{eq  3.1}). Then $\text{Lat}~T=\big\{\{0\},L_{1},L_{2},\ldots,L_{n},\cdots, \H\big\}$.
 \end{theorem} 
 \begin{proof}
Let $\M$ be a infinite-dimensional subspace of $\H$ invariant under $T$. Now, the Baire Category theorem guarantees the existence of at least an $x\in \M$ such that $x = \sum_{k=0}^{\infty} x_{k}e_{k}$ where infinitely many $x_{k}$'s are non-zero. At first, let us assume that $x_{0}\neq 0$. We will prove that $\M=\H$.\\

We have $T^{n}x = \sum_{k=0}^{\infty}x_{k}w_{k}\cdots w_{k+n-1}e_{k+n}$. Then
\begin{align*}
 \bigg\|\frac{T^{n}x}{x_{0}w_{0}\cdots w_{n-1}}-e_{n}\bigg\|^{2} &= \bigg\|\sum_{k=1}^{\infty}\frac{x_{k}w_{k}\cdots w_{k+n-1}}{x_{0}w_{0}\cdots w_{n-1}}e_{k+n}\bigg\|^{2}\\
    &= \sum_{k=1}^{\infty}\bigg(\frac{w_{k}\cdots w_{k+n-1}}{w_{0}\cdots w_{n-1}}\bigg)^{2}\bigg|\frac{x_{k}}{x_{0}}\bigg|^{2}\\
    &= \sum_{k=0}^{\infty}\bigg(\frac{w_{k+1}\cdots w_{k+n}}{w_{0}\cdots w_{n-1}}\bigg)^{2}\bigg|\frac{x_{k+1}}{x_{0}}\bigg|^{2}\\
    &= \frac{1}{w_{0}^{2}w_{1}^{2}}\sum_{k=0}^{\infty}\bigg(\frac{w_{k+1}\cdots w_{k+n}}{w_{2}\cdots w_{n-1}}\bigg)^{2}\bigg|\frac{x_{k+1}}{x_{0}}\bigg|^{2}\\
    &= \frac{w_{n}^{2}}{w_{0}^{2}w_{1}^{2}}\sum_{k=0}^{\infty}\bigg(\frac{w_{k+1}\cdots w_{k+n}}{w_{2}\cdots w_{n}}\bigg)^{2}\bigg|\frac{x_{k+1}}{x_{0}}\bigg|^{2}\\
    &\leq \frac{w_{n}^{2}\|x\|^{2}}{x_{0}^{2}w_{0}^{2}w_{1}^{2}}\sum_{k=0}^{\infty}\bigg(\frac{w_{k+2}\cdots w_{k+n}}{w_{2}\cdots w_{n}}\bigg)^{2}w_{k+1}^{2}\\
     &\leq \frac{\mu^{2} w_{n}^{2}\|x\|^{2}}{x_{0}^{2}w_{0}^{2}w_{1}^{2}}\sum_{k=0}^{\infty}\bigg(\frac{w_{k+2}\cdots w_{k+n}}{w_{2}\cdots w_{n}}\bigg)^{2}~~~(\mu=\text{sup}\{w_{n}\})\\
     &\leq \frac{\delta \mu^{2} w_{n}^{2}\|x\|^{2}}{x_{0}^{2}w_{0}^{2}w_{1}^{2}}
     = C w_{n}^{2},
\end{align*}
where $C=\frac{\delta \mu^{2} \|x\|^{2}}{x_{0}^{2}w_{0}^{2}w_{1}^{2}}$ is a constant. Since we know $\sum_{n=0}^{\infty}w_{n}^{2} < \infty$, therefore 
$$
\sum_{n=0}^{\infty} \bigg\|\frac{T^{n}x}{x_{0}w_{0}...w_{n-1}}-e_{n}\bigg\|^{2} \leq C\sum_{n=0}^{\infty}w_{n}^{2} < \infty.
$$

This implies that the orthonormal basis sequence $\{e_{n}\}_{n=0}^{\infty}$ and $\{T^{n}x/x_{0}w_{0}...w_{n-1}\}_{n=0}^{\infty}$ are quadratically close. From the Remark \ref{rem 3.2}, the set $\{T^{n}x/x_{0}w_{0}...w_{n-1}\}_{n=0}^{\infty}$ is $\omega-$independent. Now using Bari's Theorem (\cite{young1981introduction}, Thm $15$) for the pair of quadratically close $\omega-$independent sequences, we get that $\{T^{n}x/x_{0}w_{0}...w_{n-1}\}_{n=0}^{\infty}$ forms a Riesz basis of the Hilbert space $\H$. Moreover, the set $\bigvee_{n=0}^{\infty}\{T^{n}x\}\subseteq \M\subseteq \H$, and whence we conclude that $\M=\H$.

Again, if $x_{0}= 0$ and $k$ is the least natural number such that there exists an $x\in \M$ with $\langle x,e_{k}\rangle\neq 0$. So using Bari's theorem as above we can show that 
$$
\M= \bigvee_{n=0}^{\infty}\{T^{n}x\}= L_{k}.
$$
Thus every cyclic subspace of $T$ is of the form $ L_{k}$. Finally our theorem follows from the observation that span of any number of $L_{i}$'s is again an $L_{i}$.
 \end{proof}
 
 In the same paper \cite{yadav1982characterization}, the authors claimed that their result (Theorem \ref{thm 3.5}) is a generalization of Theorem 2 from \cite{nikolskii1965invariant}. But we give the following two examples to show that the two set of conditions imposed on the weight sequence are independent.
 For reader's reference, we recall that Theorem 2 from \cite{nikolskii1965invariant} gives the same characterization of invariant subspaces of $T$ as given by Theorem \ref{thm 3.5} under the assumption that the weight sequence is a square summable monotonically decreasing sequence of positive real numbers.   
 
 \begin{ex}\label{ex 3.7}
\normalfont{ Let $w_{n}=\frac{1}{n}$ for all $n\geq 1$. Then $\{w_{n}\}$ is bounded monotonically decreasing sequence of real numbers such that $\{w_n\}\in \ell^2.$ Let $T$ be the forward weighted shift on $\H$, then clearly the weights $\{w_{n}\}$ satisfy the hypothesis given in Theorem 2 from  \cite{nikolskii1965invariant}. 
 
 For each $n\in \mathbb{N}$, let us define
$$
a_{n}= \sum_{k=1}^{\infty}\bigg(\frac{w_{k+n}}{ w_{n}}\bigg)^{2}
        = \sum_{k=1}^{\infty}\bigg(\frac{n}{ n+k}\bigg)^{2}.
$$

Now for each $n\in \mathbb{N}$, let us consider the following function
$$
 f_{n}:[1,\infty)\rightarrow{\mathbb{R}},~\text{such that}~f_{n}(x)=\Big(\frac{n}{n+x}\Big)^{2}.
$$

Then clearly for each $n\in \mathbb{N}$, the function $f_{n}$ is positive valued and monotonically decreasing on $[1,\infty)$ such that $f_{n}(k)=\Big(\frac{n}{n+k}\Big)^{2}$. Then by integral test for convergence of series we get
$$
  a_{n}=  \sum_{k=1}^{\infty}\bigg(\frac{n}{ n+k}\bigg)^{2}\geq \int\limits_{1}^{\infty}f_{n}(x)dx=\frac{n^{2}}{n+1}.
$$
This implies that $a_{n}\rightarrow{\infty}$ as $n\rightarrow{\infty}$. If $\delta= \underset{m\geq2,n}{\sup}~\sum_{k=0}^{\infty}\bigg(\frac{w_{k+m}\cdots w_{k+n}}{w_{m}\cdots w_{n}}\bigg)^{2}$, then for $m=n, ~\delta= \sum_{k=1}^{\infty}(\frac{w_{k+n}}{ w_{n}})^{2}$ is not finite, and hence $\{w_n\}$ does not satisfy Condition \ref{eq 3.1} of Theorem \ref{thm 3.5}.}
 \end{ex} 
 
\begin{ex}\label{ex 3.8}
\normalfont{
 Let us define a sequence $\{w_{n}\}$ of positive real numbers as follows;
\begin{equation*}
    w_{n}= 
\begin{cases}
    \frac{1}{2^{n+1}},& \text{if} ~n~ \text{is even}\\
   \frac{1}{2^{n-1}},& \text{if} ~n~ \text{is odd}
\end{cases}
\end{equation*}

First of all, we have 
\begin{align}\label{eq  3.2}
     \frac{w_{k+r}}{w_{r}}= 
\begin{cases}
    \frac{1}{2^{k}},& \text{if} ~k,r~ \text{are even}\\
   \frac{1}{2^{k+2}},& \text{if} ~k,r~ \text{are odd}\\
   \frac{1}{2^{k}},& \text{if} ~k~\text{is even,}~r~ \text{is odd}\\
    \frac{1}{2^{k-2}},& \text{if} ~k~\text{is odd,}~r~ \text{is even}
   \end{cases}
\end{align}

Now for any pair of non-negative integers $(m,n)$, we have
\begin{align*}
    \sum_{k=0}^{\infty}\bigg(\frac{w_{k+m}\cdots w_{k+n}}{w_{m}\cdots w_{n}}\bigg)^{2} &= 
     1 + \bigg(\frac{w_{1+m}\cdots w_{1+n}}{w_{m}\cdots w_{n}}\bigg)^{2}+ \sum_{k=2}^{\infty}\bigg(\frac{w_{k+m}\cdots w_{k+n}}{w_{m}\cdots w_{n}}\bigg)^{2}\\
     &\leq 17 + \sum_{k=2}^{\infty}\bigg(\frac{w_{k+n}}{w_{n}}\bigg)^{2},~\text{using ~(\ref{eq  3.2})}\\
     & < 17+ \frac{15}{16}= \frac{272}{16}.
\end{align*}

Since this is true for any pair of non-negative integers $(m,n)$, this implies that $ \delta<\infty$. Clearly, the given sequence $\{w_{n}\}$ is bounded but not monotonically decreasing, therefore the given weights satisfy the conditions given in Theorem \ref{thm 3.5} but does not satisfy the conditions given in Theorem 2 of \cite{nikolskii1965invariant}.}
\end{ex}
 
We now give characterizations of infinite-dimensional subspaces of $\H$ that are invariant under $T^{*l}$ for $2\le l\le 3$ and infinite-dimensional subspaces that are jointly invariant under $T^{*2}$ and $T^{*3}.$ As 
noted earlier, we focus on two classes of weights considered in \cite{nikolskii1965invariant} and \cite{yadav1982characterization}. We start with a characterization of infinite-dimensional subspaces that are  invariant under $T^{*2}$ where the weights satisfy the conditions of \cite{yadav1982characterization}. We have already shown in Theorem \ref{thm 3.6} that the condition of bounded variation assumed in \cite{yadav1982characterization} is redundant, and so we do not impose this condition on weights in our work. 

\begin{theorem}\label{thm 3.9}
Let the weight sequence $\{w_{n}\}_{n \in \mathbb{N}}$ satisfy Condition (\ref{eq  3.1}). Suppose $\M$ is a proper infinite-dimensional subspace of $\H$ invariant under $T^{*2}$. Then, $\M$ has one of the following forms:
\begin{enumerate}[(i)]
\item $\M = \bigvee_{i=0}^{\infty}\{e_{2i+t}\}$ for $t=0, 1$.
\item $\M = M_{2n+1+t}\bigoplus\bigvee_{i=n+1+t}^{\infty} \{e_{2i+1-t}\}$ for some $n\geq 0$ and $t=0,1$.
\end{enumerate}
\end{theorem}

\begin{proof}
According to the given hypothesis $\M$ is infinite-dimensional subspace of $\H$, then the Baire Category theorem guarantees the existence of at least an $x\in \M$ such that $x = \sum_{i=0}^{\infty} x_{i}e_{i}$ where infinitely many $x_{i}$'s are non-zero. Now as we proceed, we will exhaust all possible combinations of elements in $\M$.

\vspace{.2cm}
\noindent\underline{\textbf{Case $1$.}} First we assume that there exists an $x\in \M$ such that $x = \sum_{i=0}^{\infty} x_{2i}e_{2i}$, where infinitely many $x_{2i}$'s are non-zero. Then, we claim that $\bigvee_{i=0}^{\infty}\{e_{2i}\}\subseteq \M$.\\

First we will show that $e_{0}\in \M$. Now for all $m\geq 1$, we have
$$
 T^{*2m}x = \sum_{i=m}^{\infty} x_{2i}w_{2i-1}\cdots {w_{2i-2m}}e_{2i-2m}.
$$

Suppose there exist a natural number $n$ such that $x_{2n}\neq 0$, then
\begin{align*}
    \bigg\|\frac{T^{*2n}x}{x_{2n}w_{2n-1}\cdots w_{0}}-e_{0}\bigg\|^{2} &= \bigg\|\sum_{i=n+1}^{\infty}\frac{x_{2i}w_{2i-1}\cdots w_{2i-2n}}{x_{2n}w_{2n-1}\cdots w_{0}}e_{2i-2n}\bigg\|^{2}\\
    &= \sum_{i=n+1}^{\infty}\bigg(\frac{w_{2i-1}\cdots w_{2i-2n}}{w_{2n-1}\cdots w_{0}}\bigg)^{2}\bigg|\frac{x_{2i}}{x_{2n}}\bigg|^{2}\\
   &= \frac{1}{w_{0}^{2}w_{1}^{2}}\sum_{i=n+1}^{\infty}\bigg(\frac{w_{2i-2n}\cdots w_{2i-1}}{w_{2}\cdots w_{2n}}\bigg)^{2}w_{2n}^{2}\bigg|\frac{x_{2i}}{x_{2n}}\bigg|^{2}\\
   &= \frac{w_{2n}^{2}}{w_{0}^{2}w_{1}^{2}}\sum_{i=0}^{\infty}\bigg(\frac{w_{2i+2}\cdots w_{2i+2n+1}}{w_{2}\cdots w_{2n}}\bigg)^{2}\bigg|\frac{x_{2i+2n+2}}{x_{2n}}\bigg|^{2}\\
   &= \frac{w_{2n}^{2}}{w_{0}^{2}w_{1}^{2}}\sum_{i=0}^{\infty}\bigg(\frac{w_{2i+2}\cdots w_{2i+2n}}{w_{2}\cdots w_{2n}}\bigg)^{2}w_{2i+2n+1}^{2}~\bigg|\frac{x_{2i+2n+2}}{x_{2n}}\bigg|^{2}\\
   & \leq \frac{\mu^{2}\delta}{w_{0}^{2}w_{1}^{2}}\sum_{i=0}^{\infty}w_{2i+2n+1}^{2}~\bigg|\frac{x_{2i+2n+2}}{x_{2n}}\bigg|^{2}.
\end{align*}

As we know $\{w_{n}\}_{n \in \mathbb{N}}\in \ell^{2}$, so for all $\epsilon > 0$ there exists an integer $J$ such that $\sum_{i= J}^{\infty}w_{2i}^{2}< \epsilon$. Also since $\{|x_{2i}|\}$ are bounded, choose $ K \geq J$ such that $x_{2K}= \underset{i\geq J}{\sup}~\{|x_{2i}|\}$. Thus
\begin{align*}
  \bigg\|\frac{T^{*2K}x}{x_{2K}w_{2K-1}\cdots w_{0}}-e_{0}\bigg\|^{2}  &\leq \frac{\mu^{2}\delta}{w_{0}^{2}w_{1}^{2}}\sum_{i=0}^{\infty}w_{2i+2K+1}^{2}~\bigg|\frac{x_{2i+2K+2}}{x_{2K}}\bigg|^{2}\\
  &\leq \frac{\mu^{2}\delta}{w_{0}^{2}w_{1}^{2}}\sum_{i=0}^{\infty}w_{2i+2K+1}^{2}
  < C\epsilon,
\end{align*}
here $C$ is a constant and this implies that $e_{0}\in \M$.

As for each $n\geq 0$, we get $y= T^{*2n}x- x_{2n}w_{2n-1}\cdots w_{0}e_{0}= \sum_{i=n+1}^{\infty}x_{2i}w_{2i-1}\cdots w_{2i-2n}e_{2i-2n}\in \M$, therefore whenever $x_{2n+2}\neq 0$, we get
\begin{align*}
   \bigg\|\frac{y}{x_{2n+2}w_{2n+1}\cdots w_{2}}-e_{2}\bigg\|^{2}  &= \bigg\|\sum_{i=n+2}^{\infty}\frac{x_{2i}w_{2i-1}\cdots w_{2i-2n}}{x_{2n+2}w_{2n+1}\cdots w_{2}}e_{2i-2n}\bigg\|^{2}\\
    &= \sum_{i=n+2}^{\infty}\bigg(\frac{w_{2i-1}\cdots w_{2i-2n}}{w_{2n+1}\cdots w_{2}}\bigg)^{2}\bigg|\frac{x_{2i}}{x_{2n+2}}\bigg|^{2}\\
   &= \frac{1}{w_{2}^{2}w_{3}^{2}}\sum_{i=n+2}^{\infty}\bigg(\frac{w_{2i-2n}\cdots w_{2i-1}}{w_{4}\cdots w_{2n+2}}\bigg)^{2}w_{2n+2}^{2}\bigg|\frac{x_{2i}}{x_{2n+2}}\bigg|^{2}\\
   &= \frac{w_{2n+2}^{2}}{w_{2}^{2}w_{3}^{2}}\sum_{i=0}^{\infty}\bigg(\frac{w_{2i+4}\cdots w_{2i+2n+3}}{w_{4}\cdots w_{2n+2}}\bigg)^{2}\bigg|\frac{x_{2i+2n+4}}{x_{2n+2}}\bigg|^{2}\\
   &= \frac{w_{2n+2}^{2}}{w_{2}^{2}w_{3}^{2}}\sum_{i=0}^{\infty}\bigg(\frac{w_{2i+4}\cdots w_{2i+2n+2}}{w_{2}\cdots w_{2n}}\bigg)^{2}w_{2i+2n+3}^{2}~\bigg|\frac{x_{2i+2n+4}}{x_{2n+2}}\bigg|^{2}.
\end{align*}

Then by the same procedure as done for $e_{0}$, we can show that $e_{2}\in \M$. Proceeding in this manner we can show that $e_{2i}\in \M$ for every $i\in \mathbb{N}$, and hence, $\bigvee_{i=0}^{\infty}\{e_{2i}\}\subseteq \M$.
Furthermore, if $\M\subseteq\bigvee_{i=0}^{\infty}\{e_{2i}\}$, then the above calculations implies that $\M= \bigvee_{i=0}^{\infty}\{e_{2i}\}$.

\vspace{.2cm}
\noindent\underline{\textbf{Case $2$.}} 
Suppose there exists an $x\in \M$ such that $x = \sum_{i=0}^{\infty} x_{2i+1}e_{2i+1}$, where infinitely many $x_{2i+1}$'s are non-zero. Then, using similar arguments as Case $1$, we can conclude that $\M= \bigvee_{i=0}^{\infty}\{e_{2i+1}\}$.
 
\vspace{.2cm}
\noindent\underline{\textbf{Case $3$.}} 
Suppose there exists an $x\in \M$ such that
\begin{equation}\label{eq  3.3}
     x 
    = \sum_{i=0}^{n} x_{2i+1}e_{2i+1}+\sum_{i=0}^{\infty} x_{2i}e_{2i},
\end{equation}
where infinitely many $x_{2i}$'s are non-zero, then we claim that $M_{2n+1}\bigoplus\bigvee_{i=n+2}^{\infty} \{e_{2i}\}\subseteq \M$.\\

We will first show that $e_{0}\in \M$. Now choose an positive integer $p>n+1$ such that $x_{2p}\neq 0$, then
$$
 \frac{T^{*2p}x}{x_{2p}w_{2p-1}\cdots w_{0}}-e_{0}= \sum_{i=p+1}^{\infty}\frac{x_{2i}w_{2i-1}\cdots w_{2i-2p}}{x_{2p}w_{2p-1}\cdots w_{0}}e_{2i-2p}.
$$

Again using the same procedure as Case $1$, we can show that $\bigvee_{i=0}^{\infty} \{e_{2i}\}\subseteq \M$.

From Equation (\ref{eq  3.3}), consider $y=x-\sum_{i=0}^{\infty}x_{2i}e_{2i}=\sum_{i=0}^{n}x_{2i+1}e_{2i+1}$. Since $\bigvee_{i=0}^{\infty} \{e_{2i}\}\subseteq \M$, then $y$ is an element in $\M$.
Now, $T^{*2n}y=x_{2n+1}w_{2n+1}\cdots w_{2}e_{1}$, and since $T^{*2n}y\in \M$, then $ e_{1}\in \M$. 

Similarly, $T^{*2(n-1)}y=x_{2n-1}w_{2n-1}\cdots w_{2}e_{1}~ +~ x_{2n+1}w_{2n+1}\cdots w_{4}e_{3}$, and since $\{T^{*2(n-1)}y,e_{1}\}\subseteq \M $, then $e_{3}\in \M$.

Proceeding in this manner, we can verify that $\bigvee_{i=0}^{n} \{e_{2i+1}\}\subseteq \M$. Now, if $\M\subseteq M_{2n+2}\bigoplus\bigvee_{i=n+2}^{\infty} \{e_{2i}\}$ for some $n\geq 0$, then $\M= M_{2n+1}\bigoplus\bigvee_{i=n+2}^{\infty} \{e_{2i}\}$.

\vspace{.2cm}
\noindent\underline{\textbf{Case $4$.}}
Suppose there exists an $x\in \M$ such that
\begin{equation*}
     x = \sum_{i=0}^{n} x_{2i}e_{2i}+\sum_{i=0}^{\infty} x_{2i+1}e_{2i+1},
\end{equation*}
where infinitely many $x_{2i+1}$'s are non-zero, then using the similar arguments as of Case $3$, one can conclude that $ \M= M_{2n}\oplus\bigvee_{i=n}^{\infty} \{e_{2i+1}\}$.

\vspace{.2cm}
\noindent\underline{\textbf{Case $5$.}} Suppose there exists an $x\in \M$ such that $x = \sum_{i=0}^{\infty} x_{i}e_{i}$, where infinitely many $x_{i}$'s are non-zero.
We will show that $\bigvee_{i=0}^{\infty}\{e_{i}\}\subseteq \M$.\\

First we will show that $e_{0}\in \M$. Now for all $m\geq 1$, we have
$$
 T^{*2m}x = \sum_{i=2m}^{\infty} x_{i}w_{i-1}\cdots {w_{i-2m}}e_{i-2m}.
$$

Suppose there exist a natual number $n$ such that $x_{2n}\neq 0$, then
\begin{align*}
    \bigg\|\frac{T^{*2n}x}{x_{2n}w_{2n-1}\cdots w_{0}}-e_{0}\bigg\|^{2} &= \bigg\|\sum_{i=2n+1}^{\infty}\frac{x_{i}w_{i-1}\cdots w_{i-2n}}{x_{2n}w_{2n-1}\cdots w_{0}}e_{i-2n}\bigg\|^{2}\\
    &= \sum_{i=2n+1}^{\infty}\bigg(\frac{w_{i-1}\cdots w_{i-2n}}{w_{2n-1}\cdots w_{0}}\bigg)^{2}\bigg|\frac{x_{i}}{x_{2n}}\bigg|^{2}\\
   &= \frac{1}{w_{0}^{2}w_{1}^{2}}\sum_{i=0}^{\infty}\bigg(\frac{w_{i+1}\cdots w_{i+2n}}{w_{2}\cdots w_{2n-1}}\bigg)^{2}\bigg|\frac{x_{i+2n+1}}{x_{2n}}\bigg|^{2}\\
   &= \frac{1}{w_{0}^{2}w_{1}^{2}}\sum_{i=0}^{\infty}\bigg(\frac{w_{i+2}\cdots w_{i+2n-1}}{w_{2}\cdots w_{2n-1}}\bigg)^{2}w_{i+2n}^{2}w_{i+1}^{2}~\bigg|\frac{x_{i+2n+1}}{x_{2n}}\bigg|^{2}\\
   & \leq \frac{\mu^{2}\delta}{w_{0}^{2}w_{1}^{2}}\sum_{i=0}^{\infty}w_{i+2n}^{2}~\bigg|\frac{x_{i+2n+1}}{x_{2n}}\bigg|^{2}.
\end{align*}
As we know $\{w_{n}\}_{n \in \mathbb{N}}\in \ell^{2}$, so for all $\epsilon > 0$ there exists an integer $J$ such that $\sum_{i=J}^{\infty}w_{i}^{2}< \epsilon$. Also since $\{|x_{i}|\}$ are bounded, choose $K \geq J$ such that $x_{2K}= \underset{i\geq 2{J}+1}{\sup}~\{|x_{i}|\}$. Thus
\begin{align*}
  \bigg\|\frac{T^{*2K}x}{x_{2K}w_{2K-1}\cdots w_{0}}-e_{0}\bigg\|^{2}  &\leq \frac{\mu^{2}\delta}{w_{0}^{2}w_{1}^{2}}\sum_{i=0}^{\infty}w_{i+2K}^{2}~\bigg|\frac{x_{i+2K+1}}{x_{2K}}\bigg|^{2}\\
  &\leq \frac{\mu^{2}\delta}{w_{0}^{2}w_{1}^{2}}\sum_{i=0}^{\infty}w_{i+2K}^{2}< C\epsilon,
\end{align*}
here $C$ is a constant and this implies that $e_{0}\in \M$.

As for each $n\geq 0,~ y=T^{*2n}x- x_{2n}w_{2n-1}\cdots w_{0}e_{0}= \sum_{i=2n+1}^{\infty}x_{i}w_{i-1}\cdots w_{i-2n}e_{i-2n}\in \M$, therefore whenever $x_{2n+1}\neq 0$ we have
\begin{align*}
   \bigg\|\frac{y}{x_{2n+1}w_{2n}\cdots w_{1}}-e_{1}\bigg\|^{2}  &= \bigg\|\sum_{i=2n+2}^{\infty}\frac{x_{i}w_{i-1}\cdots w_{i-2n}}{x_{2n+1}w_{2n}\cdots w_{1}}e_{i-2n}\bigg\|^{2}\\
    &= \sum_{i=2n+2}^{\infty}\bigg(\frac{w_{i-1}\cdots w_{i-2n}}{w_{2n}\cdots w_{1}}\bigg)^{2}\bigg|\frac{x_{i}}{x_{2n+1}}\bigg|^{2}\\
   &= \frac{1}{w_{1}^{2}}\sum_{i=2n+2}^{\infty}\bigg(\frac{w_{i-2n}\cdots w_{i-1}}{w_{2}\cdots w_{2n}}\bigg)^{2}\bigg|\frac{x_{i}}{x_{2n+1}}\bigg|^{2}\\
   &= \frac{1}{w_{1}^{2}}\sum_{i=0}^{\infty}\bigg(\frac{w_{i+2}\cdots w_{i+2n+1}}{w_{2}\cdots w_{2n}}\bigg)^{2}\bigg|\frac{x_{i+2n+2}}{x_{2n+1}}\bigg|^{2}\\
   &= \frac{1}{w_{1}^{2}}\sum_{i=0}^{\infty}\bigg(\frac{w_{i+2}\cdots w_{i+2n}}{w_{2}\cdots w_{2n}}\bigg)^{2}w_{i+2n+1}^{2}~\bigg|\frac{x_{i+2n+2}}{x_{2n+1}}\bigg|^{2}.
\end{align*}
Then applying the same procedure as done for $e_{0}$, it can be shown that $e_{1}\in \M$. Proceeding in this manner, we can show that $e_{i}\in \M$ for every $i\in \mathbb{N}$, and hence $ \bigvee_{i=0}^{\infty}\{e_{i}\}\subseteq \M$. Moreover, $\M\subseteq\H=\bigvee_{i=0}^{\infty}\{e_{i}\}$, therefore $ \M= \H$.

This completes the proof.
\end{proof}

Now we will work with second class of weight as considered in \cite{nikolskii1965invariant}. Let $\{w_{n}\}_{n\in \mathbb{N}}$ be a sequence of non-zero real numbers such that 
\begin{equation}\label{eq  3.4}
  w_{n+1}\leq w_{n} ~ \text{for all}~n\geq 0 \quad \text{and}\quad \{w_{n}\}_{n\in \mathbb{N}}\in \ell{^2}.         
\end{equation}
   
\begin{theorem}\label{thm 3.10}
Let the weight sequence $\{w_{n}\}_{n \in \mathbb{N}}$ satisfy Condition (\ref{eq  3.4}). Suppose $\M$ is a proper infinite-dimensional subspace of $\H$ invariant under $T^{*2}$. Then, $\M$ has one of the following forms:
\begin{enumerate}[(i).]
\item $\M = \bigvee_{i=0}^{\infty}\{e_{2i+t}\}$ for $t=0, 1$.
\item $\M = M_{2n+1+t}\oplus\bigvee_{i=n+1+t}^{\infty} \{e_{2i+1-t}\}$ for some $n\geq 0$ and $t=0, 1$.
\end{enumerate}
\end{theorem}
\begin{proof}
As in the proof of Theorem \ref{thm 3.9}, here also we divide the proof in five cases based on the non-zero coefficients of elements of $\M$.\\

\noindent\underline{\textbf{Case $1$.}} First we assume that there exists an $x\in \M$ such that $x = \sum_{i=0}^{\infty} x_{2i}e_{2i}$, where infinitely many $x_{2i}$'s are non-zero, then we claim that $ \bigvee_{i=0}^{\infty}\{e_{2i}\}\subseteq \M$.\\

First we will show that $e_{0}\in \M$. Now for all $m\geq 1$, we have
$$
 T^{*2m}x = \sum_{i=m}^{\infty} x_{2i}w_{2i-1}\cdots {w_{2i-2m}}e_{2i-2m}.
$$

Suppose there exist a natural number $n$ such that $x_{2n}\neq 0$, then
$$
\bigg\|\frac{T^{*2n}x}{x_{2n}w_{2n-1}\cdots w_{0}}-e_{0}\bigg\|^{2} = \bigg\|\sum_{i=n+1}^{\infty}\frac{x_{2i}w_{2i-1}\cdots w_{2i-2n}}{x_{2n}w_{2n-1}\cdots w_{0}}e_{2i-2n}\bigg\|^{2}.
$$

Since $\{w_{n}\}_{n=0}^{\infty}$ is a monotonically decreasing sequence, then $\frac{w_{2i-j-1}}{w_{2n-j}}\leq 1,~i\geq n+1,~ 1\leq j\leq 2n-1$. Now, we have 
\begin{align*}
 \bigg\|\sum_{i=n+1}^{\infty}\frac{x_{2i}w_{2i-1}\cdots w_{2i-2n}}{x_{2n}w_{2n-1}\cdots w_{0}}e_{2i-2n}\bigg\|^{2} &=\sum_{i=n+1}^{\infty}\bigg|\frac{x_{2i}w_{2i-1}\cdots w_{2i-2n}}{x_{2n}w_{2n-1}\cdots w_{0}}\bigg|^{2}\\
    & \leq \sum_{i=n+1}^{\infty}\bigg|\frac{x_{2i}}{x_{2n}}\bigg|^{2}\bigg(\frac{w_{2i-1}}{w_{0}}\bigg)^{2}.
\end{align*}

As we know $\{w_{n}\}_{n \in \mathbb{N}}\in \ell^{2}$, so for all $\epsilon > 0$, there exists an positive integer $J$ such that $\sum_{i= J}^{\infty}\frac{w_{2i-1}^{2}}{w_{0}^{2}}< \epsilon$ and since $\{|x_{2i}|\}$ are bounded, choose an positive integer $K \geq J$ such that $x_{2K}= \underset{i\geq J}{\max}\{|x_{2i}|\}$.\\
Thus, we have
$$
  \bigg\|\frac{T^{*2K}x}{x_{2K}w_{2K-1}\cdots w_{0}}-e_{0}\bigg\|^{2}  \leq\sum_{i=K+1}^{\infty}\bigg|\frac{x_{2i}}{x_{2n}}\bigg|^{2}\bigg(\frac{w_{2i-1}}{w_{0}}\bigg)^{2} < \epsilon ~\text{as}~K\to \infty.
$$

Since the sequence $\{T^{*2K}x/x_{2K}w_{2K-1}\cdots w_{0}\}$ converges in $\M$ to $e_{0}$, this implies that $ e_{0}\in \M$.

As for each $n\geq 0,~ y= T^{*2n}x- x_{2n}w_{2n-1}\cdots w_{0}e_{0}= \sum_{i=n+1}^{\infty}x_{2i}w_{2i-1}\cdots w_{2i-2n}e_{2i-2n}\in \M$, therefore whenever $x_{2n+2}\neq 0$, we have
$$
  \bigg\|\frac{y}{x_{2n+2}w_{2n+1}\cdots w_{2}}-e_{2}\bigg\|^{2}  = \bigg\|\sum_{i=n+2}^{\infty}\frac{x_{2i}w_{2i-1}\cdots w_{2i-2n}}{x_{2n+2}w_{2n+1}\cdots w_{2}}e_{2i-2n}\bigg\|^{2}.
$$

Then by the same procedure as done for $e_{0}$, we can show that $e_{2}\in \M$. Proceeding in this manner we can show that $e_{2i}\in \M$ for every $i\in \mathbb{N}$, and hence $ \bigvee_{i=0}^{\infty}\{e_{2i}\}\subseteq \M$. Moreover, if $\M\subseteq\bigvee_{i=0}^{\infty}\{e_{2i}\}$ such that $T^{*2}\M\subseteq \M$, then $\M= \bigvee_{i=0}^{\infty}\{e_{2i}\}$.

The rest of the following cases can be dealt with using arguments similar to the ones used in Case 1. For completion, we note the cases and the various forms that $\M$ assume in their.

\vspace{.2cm}
\noindent\underline{\textbf{Case $2$.}} 
Suppose there exists an $x\in \M$ such that $x = \sum_{i=0}^{\infty} x_{2i+1}e_{2i+1}$, where infinitely many $x_{2i+1}$'s are non-zero. Then, $\M= \bigvee_{i=0}^{\infty}\{e_{2i+1}\}$.
 
\vspace{.2cm}
\noindent\underline{\textbf{Case $3$.}} 
Suppose there exists an $x\in \M$ such that $x = \sum_{i=0}^{n} x_{2i+1}e_{2i+1}+\sum_{i=0}^{\infty} x_{2i}e_{2i}$, where infinitely many $x_{2i}$'s are non-zero. Then,  $\M=M_{2n+1}\bigoplus\bigvee_{i=n+2}^{\infty} \{e_{2i}\}$.

\vspace{.2cm}
\noindent\underline{\textbf{Case $4$.}}
Suppose there exists an $x\in \M$ such that $x = \sum_{i=0}^{n} x_{2i}e_{2i}+\sum_{i=0}^{\infty} x_{2i+1}e_{2i+1}$, where infinitely many $x_{2i+1}$'s are non-zero. Then, $ \M= M_{2n}\oplus\bigvee_{i=n}^{\infty} \{e_{2i+1}\}$.

\vspace{.2cm}
\noindent\underline{\textbf{Case $5$.}} Suppose there exists an $x\in \M$ such that $x = \sum_{i=0}^{\infty} x_{i}e_{i}$, where infinitely many $x_{i}$'s are non-zero. Then, $\M=\H$.

This completes the proof.
\end{proof}

\begin{theorem}\label{thm 3.11}
Let the weights satisfy either Condition (\ref{eq  3.1}) or Condition (\ref{eq  3.4}) and $\M\subseteq \H$ be a non-zero proper infinite-dimensional subspace such that $T^{*3}\M\subseteq \M$. Then $\M$ can have one of the following forms:
\begin{enumerate}[(i)]
\item $\M = \bigvee_{i=0}^{\infty}\{e_{3i+t}\}$ for $t\in \{0,1,2\}$.
\item $\M = \bigvee_{i=0}^{n}\{e_{3i+r}\}\bigoplus \bigvee_{i=0}^{m} \{e_{3i+s}\}$ for $r,s\in \{0,1,2\}$, $r\neq s$ such that $n$ and $m$ cannot be both finite.
\item $\M = M_{3l+2}\bigoplus\bigvee_{i=l+1}^{n} \{e_{3i+r}\}\bigoplus \bigvee_{i=l+1}^{m} \{e_{3i+s}\}$ for $r,s\in \{0,1,2\}$, $r\neq s$, and $l\geq 0$ such that $n$ and $m$ cannot be both finite.
\end{enumerate}
\end{theorem}
 \begin{proof}
 If the weights satisfy conditions (\ref{eq  3.1}) or (\ref{eq  3.4}), then we obtain the desired forms for $\M$ by using the similar techniques as we used in the proof of Theorem \ref{thm 3.9} and Theorem \ref{thm 3.10} respectively. 
 \end{proof}

Our final result of this section is a characterization of infinite-dimensional subspaces of $\H$ that are jointly invariant under $T^{*2}$ and $T^{*3}.$ 

\begin{theorem}\label{thm 3.12}
Let the weights satisfy either Condition (\ref{eq  3.1}) or Condition (\ref{eq  3.4}) and $\M\subseteq \H$ be an infinite-dimensional joint-invariant subspace of $T^{*2}$ and $T^{*3}$. Then $\M=\H$. 
\end{theorem}
\begin{proof}
Since $\M$ is invariant under both $T^{*2}$ and $T^{*3}$, therefore it is easy to see that $T^{*j}\M\subseteq \M$ for all $j\geq 2$. As $\M$ is infinite-dimensional, there exist an $x\in \M$ such that $x = \sum_{i=0}^{\infty} x_{i}e_{i}$, where infinitely many $x_{i}$'s are non-zero. Now, depending on the conditions on the weights (\ref{eq  3.1}) or (\ref{eq  3.4}), we can use the same proof of \cite{nikolskii1965invariant} or \cite{yadav1982characterization} respectively to show that $\bigvee_{i=0}^{\infty}\{e_{i}\}\subseteq \M$. Hence, $\M=\H$. This completes the proof.
\end{proof}
   

\section{Final remarks and results on Unicellular operators}\label{sec 4}
\begin{remark}
 For $i\geq 2$, the technique we used for characterizing an n-dimensional non-cyclic $T^{*i}$-invariant subspace involves partitioning $n$ into $k~(2\leq k\leq i)$ number of $T^{*i}$-invariant cyclic subspaces. Even for $i=3$, this technique resulted in many tedious calculations in the proof of Theorem \ref{ thm 2.8}. Now if $i$ gets larger, then the choices of partitioning $n$ into $k~(2\leq k\leq i)$ numbers also increases, due to which it does not seem feasible to apply the same technique for $i\geq 4$. It will be interesting to find new techniques to tackle the problem when the number $i$ gets larger.
\end{remark}
\begin{remark}
 As demonstrated in Sections \ref{sec 2} and \ref{sec 3}, weights have no bearing on our characterization of finite-dimensional invariant subspaces of the operators $T^{*2}$ and $T^{*3}$, but they are crucial for the infinite-dimensional cases. Interestingly, the lattice structure of the invariant subspaces of $T^{*2}$ and $T^{*3}$ is the same (Theorem \ref{thm 3.9}, \ref{thm 3.10} and \ref{thm 3.11}) for two independent classes of weights that we have considered. There are numerous more sufficient conditions on the weights $\{w_{n}\}_{n\in \mathbb{N}}$ for the weighted shift to be unicellular, see  \cite{harrison1971unicellularity, herrero1990unicellular, kang1992unicellularity, kang1994study, joo1995unicellularity, radjavi2003invariant, shields1974weighted, yadav1982invariant, yousefi2001unicellularity}. So the first critical question is whether the lattice structure of the invariant subspaces of $T^{*2}$ and $T^{*3}$ remains the same for these other weights also? And, if the answer to this question is affirmative, then developing a single strategy that works for all these weight classes will be an extremely interesting problem to work on.
\end{remark}

 For the forward weighted shift $T$, it is easy to deduce that operators $T^{2}$ or $T^{3}$ are not unicellular. In \cite{kang1994study} and \cite{joo1995unicellularity}, it is proved that for some unicellular forward weighted shifts $T$ and polynomials $p$, the operators $p(T)$ are also unicellular. In \cite{kang1994study}, the authors proved that if $\{w_{n}\}$ is monotonically decreasing and converges to $0$ such that $\sum_{n=0}^{\infty}n^{2}w_{n}^{2}<\infty$, then $T$ is unicellular and the operators $T(I+T)^{m-1}, ~\sum_{i=1}^{n}T^{i},~\text{and}~\sum_{i=1}^{n}i^{m}T^{i}$ for $m,n\in \mathbb{N}$ are also unicellular. Working on the similar kind of problems, the authors in \cite{joo1995unicellularity} proved that if $w_{n}=r^{n}$ for some $0<r<1$, then the operator $T+T^{2}$ remains unicellular. This motivated us to ask if $T$ is a given unicellular operator, then for which polynomials $p$ and analytic functions $f$ the operators $p(T)$ and $f(T)$ are also unicellular?

We have a partial answer to this problem. To give it, we need the following well-established result from \cite{radjavi2003invariant}.
\begin{theorem}[\cite{radjavi2003invariant}, Theorem 2.14]\label{thm 4.3}
If $f$ is analytic and one-to-one on an open set containing $\eta(\sigma(A))$, then \textit{Lat} $A$= \textit{Lat} $f(A)$.
\end{theorem}

We use it to obtain the following interesting result:
\begin{cor}\label{cor 4.4}
Let $A$ be a quasinilpotent unicellular operator. Let $f$ be an analytic function on an open set containing the origin such that $f'(0)\neq 0$. Then $f(A)$ is also unicellular.
\end{cor}
\begin{proof}
It is given that $\sigma(A)=\{0\}$, then the resolvent set $\rho(A)=\mathbb{C} \setminus \{0\}$. Since the resolvent of the operator $A$ contains no bounded components, then $\eta(\sigma(A))=\{0\}$.

In complex analysis, it is a routine exercise to show that if the derivative of an analytic function is non-zero at a point, say $z_{0}$, then there exists a neighbourhood around $z_{0}$ where the function is one to one. According to the given hypothesis, let $f$ be a analytic function on an open set containing the origin and $f'(0)\neq 0$, then there exist a set around the origin, say $U$, such that $f|U$ is one to one. Moreover, $\eta(\sigma(A))\subset U$, then using Theorem \ref{thm 4.3}, we have Lat $A=$ Lat $f(A)$. Hence $f(A)$ is also unicellular.
\end{proof}

Note that for $m,n\in \mathbb{N}$, the functions $f(z)=z(1+z)^{m-1}, ~g(z)=\sum_{i=1}^{n}z^{i},~\text{and}~h(z)=\sum_{i=1}^{n}i^{m}z^{i}$ all satisfy the hypotheses of Corollary \ref{cor 4.4}. Also, the forward weighted shifts considered in \cite{kang1994study} and \cite{joo1995unicellularity} are unicellular and therefore are quasinilpotent. Hence, the results from \cite{kang1994study} and \cite{joo1995unicellularity} follow from our Corollary \ref{cor 4.4}. As a result, Corollary \ref{cor 4.4} not only extends the results of \cite{kang1994study} and \cite{joo1995unicellularity} to the entire class of unicellular forward weighted shifts, but it also provides a simpler proof of their results.

Finally, we note that the conditions of Corollary \ref{cor 4.4} are not necessary for $p(T)$ to be unicellular. For example, if $V$ is a Volterra operator
$$
V:L_{q}[0,1]\rightarrow{L_{q}[0,1]},~q\in [1,\infty)\quad\text{such that}\quad (Vf)(x)=\int_{0}^{x}f(x)dx.
$$
Then $V$ is unicellular with $\sigma(V)= \{0\}$. Moreover, $V^{n}$ is also unicellular for all $n\in \mathbb{N}$ (see \cite{gohberg1967theory}) but $p(z)=z^{n}$ does not satisfy the  conditions of Corollary \ref{cor 4.4}.

\subsection*{Acknowledgements} The first and third authors
thank the Mathematical Sciences Foundation,
Delhi for support and facilities needed to complete
the present work. The research of first author is supported by the Mathematical Research Impact Centric Support (MATRICS) grant, File No: MTR/2017/000749, by the Science and Engineering Research Board (SERB), Department of Science \& Technology (DST), Government of India. 

\bibliographystyle{plain}
\bibliography{bibliography.bib}

\begin{thebibliography}{10}

\bibitem{beurling1949two}
A.~Beurling.
\newblock On {T}wo {P}roblems {C}oncerning {L}inear {T}ransformations in
  {H}ilbert {S}pace.
\newblock {\em Acta Mathematica}, 81:239--255, 1949.

\bibitem{donoghue1957lattice}
W.~F. Donoghue.
\newblock The {L}attice of {I}nvariant {S}ubspaces of a {C}ompletely
  {C}ontinuous {Q}uasi-nilpotent {T}ransformation.
\newblock {\em Pacific Journal of Mathematics}, 7(2):1031--1035, 1957.

\bibitem{fang1992rosenthal}
L.~Fang.
\newblock On {R}osenthal-{S}hield's {P}roblem.
\newblock {\em Acta Mathematica Sinica}, 8(2):189--203, 1992.

\bibitem{gohberg1967theory}
I.~Gohberg and M.~G. Krein.
\newblock Theory of {V}olterra {O}perators in {H}ilbert {S}pace and {I}ts
  {A}pplications (transl.).
\newblock {\em Amer. Math. Soc. Providence}, 1967.

\bibitem{harrison1971unicellularity}
K.~J. Harrsion.
\newblock On the {U}nicellularity of {W}eighted {S}hifts.
\newblock {\em Journal of the Australian Mathematical Society}, 12(3):342--350,
  1971.

\bibitem{herrero1990unicellular}
D.~A. Herrero.
\newblock A {U}nicellular {U}niversal {Q}uasinilpotent {W}eighted {S}hift.
\newblock {\em Proceedings of the American Mathematical Society},
  110(3):649--652, 1990.

\bibitem{hoffman2014banach}
K.~Hoffman.
\newblock {\em \upshape Banach Spaces of Analytic Functions}.
\newblock \itshape Courier Corporation, \upshape 2014.

\bibitem{kang1992unicellularity}
J.~H. Kang.
\newblock {O}n the {U}nicellularity of {S}ome {O}perators.
\newblock {\em Kyungpook Mathematical Journal}, 32(3):421--434, 1992.

\bibitem{kang1994study}
J.~H. Kang and G.~H. Baik.
\newblock A {S}tudy on the {U}nicellularity of {S}ome {L}ower {T}riangular
  {O}perators.
\newblock {\em Tsukuba Journal of Mathematics}, 18(1):69--78, 1994.

\bibitem{joo1995unicellularity}
J.~H. Kang and S.~J. Young.
\newblock On the {U}nicellularity of {A}n {O}perator.
\newblock {\em Communications of the Korean Mathematical Society},
  10(4):907--916, 1995.

\bibitem{LATA2022126184}
S.~Lata, S.~Pokhriyal, and D.~Singh.
\newblock Multivariable sub-{H}ardy {H}ilbert {S}paces {I}nvariant under the
  {A}ction of n-tuple of {F}inite {B}laschke {F}actors.
\newblock {\em Journal of Mathematical Analysis and Applications}, 512(2),
  2022.

\bibitem{article}
S.~Lata and D.~Singh.
\newblock A {C}lass of sub-{H}ardy {H}ilbert {S}paces {A}ssociated with
  {W}eighted {S}hifts.
\newblock {\em Houston Journal of Mathematics}, 44:301--308, 2018.

\bibitem{nikolskii1965invariant}
N.~K. Nikolskii.
\newblock Invariant {S}ubspaces of {C}ertain {C}ompletely {C}ontinuous
  {O}perators.
\newblock {\em Vestnik Leningrad Univ}, 20:68--77, 1965.

\bibitem{nikol1967invariant}
N.~K. Nikolskii.
\newblock Invariant {S}ubspaces of {W}eighted {S}hift {O}perators.
\newblock {\em Mathematics of the USSR-Sbornik}, 3(2):159--176, 1967.

\bibitem{nikol1968basicity}
N.~K. Nikolskii.
\newblock Basicity and {U}nicellularity of {W}eighted {S}hift {O}perators.
\newblock {\em Izvestiya Rossiiskoi Akademii Nauk. Seriya Matematicheskaya},
  32(5):1123--1137, 1968.

\bibitem{radjavi2003invariant}
H.~Radjavi and P.~Rosenthal.
\newblock {\em \upshape Invariant Subspaces}.
\newblock \itshape Courier Corporation, \upshape 2003.

\bibitem{shields1974weighted}
A.~L. Shields.
\newblock {W}eighted {S}hift {O}perators and {A}nalytic {F}unction {T}heory.
\newblock {\em Math. Surveys}, 13:49--128, 1974.

\bibitem{yadav1982invariant}
B.~S. Yadav and S.~Chatterjee.
\newblock Invariant {S}ubspace {L}attices of {L}ambert's {W}eighted {S}hifts.
\newblock {\em Journal of the Australian Mathematical Society}, 33(1):135--142,
  1982.

\bibitem{yadav1982characterization}
B.~S. Yadav and S.~Chatterjee.
\newblock On a {C}haracterization of {I}nvariant {S}ubspace {L}attices of
  {W}eighted {S}hifts.
\newblock {\em Proceedings of the American Mathematical Society},
  84(4):492--496, 1982.

\bibitem{yakubovich1985invariant}
D.~V. Yakubovich.
\newblock Invariant {S}ubspaces of {W}eighted {S}hift {O}perators.
\newblock {\em Zapiski Nauchnykh Seminarov POMI}, 141:100--143, 1985.

\bibitem{young1981introduction}
R.~M. Young.
\newblock {\em \upshape An {I}ntroduction to {N}on-{H}armonic {F}ourier
  {S}eries}.
\newblock \itshape Academic press, \upshape 1981.

\bibitem{yousefi2001unicellularity}
B.~Yousefi.
\newblock {U}nicellularity of the {M}ultiplication {O}perator on {B}anach
  {S}paces of {F}ormal {P}ower {S}eries.
\newblock {\em Studia Mathematica}, 147(3):201--209, 2001.

\end{thebibliography}
\end{document}